\documentclass[12pt]{amsart}

\usepackage{macros_trianguline}

\calclayout

\title{Tangent spaces on the trianguline variety at companion points}

\author{Seginus Mowlavi}
\address{Université Paris-Saclay, CNRS, Laboratoire de mathématiques d’Orsay, 91405, Orsay, France}

\begin{document}

\begin{abstract}
Many results about the geometry of the trianguline variety have been obtained by Breuil-Hellmann-Schraen. Among them, using geometric methods, they have computed a formula for the dimension of the tangent space of the trianguline variety at dominant crystalline generic points, which has a conjectural generalisation to companion (\emph{i.e.\ }non-dominant) points. In an earlier work, they proved a weaker form of this formula under the assumption of modularity using arithmetic methods.  We prove a generalisation of a result of Bellaïche-Chenevier in $p$-adic Hodge theory and use it to extend the arithmetic methods of Breuil-Hellmann-Schraen to a wide class of companion points.
\end{abstract}

\maketitle

\tableofcontents

\numberwithin{equation}{section}

\section{Introduction}

Let $p$ be a prime number. Inspired by Kisin's study \cite{kisin03} of $p$-adic overconvergent eigenforms, and building on the notion of trianguline representation from Colmez \cite{colmezTriangulines}, Hellmann \cite{hellmann} followed by Hellmann-Schraen \cite{hellmannSchraen} introduced a rigid analytic variety parametrising triangulations of framed deformations of a fixed residual representation $\bar{r}$, called the trianguline variety $\Xtri$. Hellmann used this construction to advance the study of eigenvarieties by Bellaïche-Chenevier \cite{bellaicheChenevier}. Breuil-Hellmann-Schraen focused on the geometry of the trianguline variety in a series of papers \cite{bhs1}, \cite{bhs2}, \cite{bhs3}, especially the local geometry at crystalline generic points. They derived a number of consequences, among which results concerning classicality of overconvergent forms, existence of certain eigenforms called companion eigenforms, and the existence of companion constituents (we do not focus on automorphic representations in this work, but this features prominently in the $p$-adic local Langlands program; see for instance \cite{breuilEmerton}, \cite{bergdall}, \cite{colmezDospinescu} and \cite{colmezDospinescuPaskunas}).

In particular, Breuil-Hellmann-Schraen prove an upper bound for the dimension of the tangent space of $\Xtri$ at a crystalline generic point, and conjecture that it is exact. They also show, for strictly dominant points (\emph{i.e.\ }``classical'' points), that this conjecture is implied by classical modularity lifting conjectures.

In this paper, we generalise a result of Bellaïche-Chenevier \cite{bellaicheChenevier} about refinements of crystalline representations and associated triangulations. This allows us to adapt the method of Breuil-Hellmann-Schraen \cite{bhs2} to prove their conjecture about the dimension of the tangent space of $\Xtri$, assuming the modularity lifting conjectures, at many points which are not classical (\emph{i.e.\ }``companion'' points).

To state our results, we briefly introduce the relevant objects. We fix an integer $n>0$ and a representation $\bar{r}\colon\G_K\to\GL_n(k_L)$ from the absolute Galois group of $K$ to the residual field of $L$, where $K$ and $L$ are finite extensions of $\Q_p$. We assume that $p\nmid2n$ and that $\Sigma\coloneqq\Hom_{\Q_p}(K,L)$ has $[K:\Q_p]$ elements.

Let $\Xtri$ be the trianguline variety: it is a rigid analytic variety over $L$ defined as the Zariski-closure of the set $\Utri$ of points $(r,\underline{\delta})$ where $r$ is a framed deformation of $\bar{r}$ and $\underline{\delta}$ is a locally $\Q_p$-analytic character of $(K^\times)^n$ such that $r$ is trianguline of parameter $\underline{\delta}$ in the sense of Colmez (see \emph{e.g.\ }\cite[Dfn.\ 3.3.8]{bhs3}). We fix a crystalline point $x=(r,\underline{\delta})\in\Xtri$, \emph{i.e.\ }a point such that $r$ is a crystalline representation.

We assume that $r$ is generic, which means that (i) for each $\tau\in\Sigma$, the $\tau$-Hodge-Tate weights of $r$ are all distinct, and (ii) $\varphi/\varphi'\notin\{1,p^{f(K/\Q_p)}\}$ for any two eigenvalues $\varphi,\varphi'$ of the linearised Frobenius $\Phi$ on $\Dcris(r)$ (with multiplicities), where $f(K/\Q_p)$ is the inertia degree. For each $\tau\in\Sigma$, we write the $\tau$-Hodge-Tate weights of $r$ in increasing order $h_{\tau,1}<\ldots<h_{\tau,n}$. Then $\underline{\delta}$ is determined by the data of (i) an ordering $\underline{\varphi}=(\varphi_1,\ldots,\varphi_n)$ of the eigenvalues of $\Phi$ and (ii) for each $\tau\in\Sigma$, a permutation $w_\tau$ in the symmetric group $\Scal_n$. More precisely, one can write $\underline{\delta}=z^{w(\hbold)}\nr(\underline{\varphi})$, where the character $\nr(\underline{\varphi})=(\nr(\varphi_i))_i$ of $(K^\times)^n$ is defined by $\nr(\varphi_i)|_{\intring_K^\times}=1$ and $\nr(\varphi_i)(\varpi)=\varphi_i$ for a uniformiser $\varpi$ of $K$, and where $w\coloneqq(w_\tau)\in(\Scal_n)^\Sigma$ and $w(\hbold)\coloneqq\big(h_{\tau,w_\tau^{-1}(i)}\big)_{\tau,i}\in(\Z^n)^\Sigma$. Furthermore, there is a unique $w_\sat\in(\Scal_n)^\Sigma$ such that the point $x_\sat\coloneqq\left(r,z^{w_\sat(\hbold)}\nr(\underline{\varphi})\right)$ lies in $\Utri$.

Writing $G\coloneqq\GL_{n,\Q_p}$, note that $(\Scal_n)^\Sigma$ is the Weyl group of the reductive group $\prod_{\tau\in\Sigma}G$ over $\Q_p$. Set $B\subseteq G$ to be the Borel subgroup of to upper triangular matrices; it determines a product Bruhat order $\preceq$ on $(\Scal_n)^\Sigma$. It is known that $w_\sat\preceq w$ \cite[Thm.\ 1.7]{bhs3}.

The main result of this paper is the following:

\begin{thm}[Theorem \ref{thmMain}] \label{thmIntroMain}
Assume Conjecture \ref{conjBreuilMezard}. Let $x\in\Xtri$ be crystalline generic, such that the associated pair $(w_\sat,w)$ in $(\Scal_n)^\Sigma$ is good in the sense of Definition \ref{dfnGoodPair}. Then
\begin{equation} \label{eqDimensionIntro}
	\dim_{k(x)}T_{\Xtri,x} = \dim\Xtri - d_{ww_\sat^{-1}} + \sum_{\tau\in\Sigma}\dim_{\Q_p}T_{\overline{(Bw_\tau B/B)},w_{\sat,\tau}B} - \length(w_\sat)
\end{equation}
where $k(x)$ is the residue field of $\Xtri$ at $x$, $\overline{Bw_\tau B/B}$ is the Schubert variety associated to $w_\tau$ in the flag variety $G/B$ and we refer to the paragraph preceding Definition \ref{dfnGoodPair} for the definition of $d_{ww_\sat^{-1}}$.
\end{thm}

In the case where $w=w_0$ is the maximal element of $(\Scal_n)^\Sigma$ for the Bruhat order, \eqref{eqDimensionIntro} is proved in \cite{bhs3} (without assuming Conjecture \ref{conjBreuilMezard}). Our proof of Theorem \ref{thmIntroMain} is based on a method from \cite{bhs2}, where a weak version (also for $w=w_0$ only) is proved. This method consists in finding two subspaces $T_1$ and $T_2$ of $T_{\Xtri,x}$ and computing $\dim(T_1+T_2)=\dim{T_1}+\dim{T_2}-\dim(T_1\inter T_2)$ in the following way.

We first use the following construction due to \cite[\S2.2]{bhs2} and \cite{kisin08}. Denoting by $\hbold=(h_{\tau,1}<\ldots<h_{\tau,n})_\tau\in(\Z^n)^\Sigma$ the Hodge-Tate weights of $r$, let $\Wcrtilde$ be the reduced rigid analytic space over $L$ parametrising the data of a framed deformation $r'$ of $\bar{r}$ which is crystalline generic with Hodge-Tate weights $\hbold$, together with a refinement $\underline{\varphi'}=(\varphi'_1,\ldots,\varphi'_n)$ of $r'$ (\emph{i.e.\ }an ordering of the eigenvalues of the linearised Frobenius $\Phi'$ of $r'$). There is a smooth map of $L$-rigid analytic spaces $h\colon\Wcrtilde\to\prod_{\tau\in\Sigma}(G/B)^\rig$ which sends a point $(r',\underline{\varphi'})$ of $\Wcrtilde$ to the Hodge filtration seen as a flag on $\DdR(r')$ relative to the basis induced by the refinement $\underline{\varphi'}$. Recall the Schubert decomposition of the flag variety $\prod_{\tau\in\Sigma}G/B=\coprod_{w'\in(\Scal_n)^\Sigma}\prod_{\tau\in\Sigma}Bw'_\tau B/B$. For $w'\in(\Scal_n)^\Sigma$, we consider the closed rigid subspace $\Wcrtildew{w'}$ of $\Wcrtilde$ defined by $\Wcrtildew{w'}\coloneqq h^{-1}\big(\prod_{\tau\in\Sigma}\overline{Bw'_\tau B/B}^\rig\big)$. There is a closed immersion of rigid analytic spaces $\iota_{\hbold,w'}\colon\Wcrtildew{w'}\to\Xtri$ defined by $\iota_{\hbold,w'}(r',\underline{\varphi'})=\left(r',z^{w'(\hbold)}\nr(\underline{\varphi'})\right)$. We have $(r,\underline{\varphi})\in\Wcrtildew{w}$ and $\iota_{\hbold,w}(r,\underline{\varphi})=x$; we then set $T_1$ to be the image of the injection $T_{\Wcrtildew{w},(r,\underline{\varphi})}\inj T_{\Xtri,x}$ induced by $\iota_{\hbold,w}$.

We next define a rigid automorphism $\jmathauto\colon(r',\underline{\delta'})\longmapsto\left(r',z^{w(\hbold)-w_\sat(\hbold)}\underline{\delta'}\right)$ which maps $x_\sat$ to $x$, and we want $\jmathauto$ to map a locally closed subvariety of $\Utri$ containing $x_\sat$ into $\Xtri$. We use the weight morphism $\omega\colon\Xtri\to\Wcal_L^n$ defined by $\omega(r',\underline{\delta'})=\underline{\delta'}|_{(\intring_K^\times)^n}$, where $\Wcal_L$ is the rigid analytic space over $L$ parametrising continuous characters of $\intring_K^\times$, and cut out a closed subspace of characters $\Wcal^n_{w,w_\sat,\hbold,L}\subset\Wcal_L^n$ with the equations $\wt_\tau(\eta_{w_{\sat,\tau}(i)}\eta_{w_\tau(i)}^{-1})=h_{\tau,i}-h_{\tau,w_{\sat,\tau}^{-1}w_\tau(i)}$. Then, using the patched eigenvariety $\Xrho$ of Breuil-Hellmann-Schraen \cite[\S3]{bhs1} (where $\bar{\rho}$ is a globalisation of $\bar{r}$ given by a result of Emerton-Gee \cite[Theorem 1.2.3]{emertonGee19}) and assuming Conjecture \ref{conjBreuilMezard}, we show that if $(w_\sat,w)$ is a good pair of $(\Scal_n)^\Sigma$, then there exists an open neighborhood $U_{x_\sat}$ of $x_\sat$ in $\Utri$ such that $\jmathauto$ induces a Zariski-closed embedding of reduced rigid $L$-analytic spaces $\jmathauto\colon\overline{U_{x_\sat}\times_{\Wcal^n_L}\Wcal^n_{w,w_\sat,\hbold,L}}\inj\Xtri$. We then set $T_2$ to be the image of the injection $T_{U_{x_\sat}\times_{\Wcal^n_L}\Wcal^n_{w,w_\sat,\hbold,L},x_\sat}\inj T_{\Xtri,x}$ induced by $\jmathauto$.

It is quite straightforward to compute $\dim{T_2}=\dim{\Xtri}-d_{ww_\sat^{-1}}$. The computation of $\dim{T_1}$ and $\dim(T_1\inter T_2)$ is where our approach differs from that of \cite{bhs2}. We prove, using Berger's dictionary between cristalline $(\varphi,\Gamma_K)$-modules over the Robba ring and filtered $\varphi$-modules (\cite[Théorème A]{berger08}, \cite[Théorème 3.6]{berger02}), the following generalisation of \cite[Prop.\ 2.4.1]{bellaicheChenevier}.

\begin{prop}[Proposition \ref{propLikeBellaicheChenevier}] \label{propIntroLikeBellaicheChenevier}
Let $A$ be a local artinian $L$-algebra with residue field $L$. Let $r_A\colon\G_K\to\GL_n(A)$ be a regular crystalline representation of rank $n\in\Z_{>0}$ over $A$, with Hodge-Tate weights $\hbold\in(\Z^n)^\Sigma$. Assume that there exist $\kbold\in\Z^\Sigma$ and $\underline{\varphi}_A\in(A^\times)^n$, with the $\varphi_{A,i}$ having pairwise distinct reductions in $L$, such that $r_A$ is trianguline of parameter $z^\kbold\nr(\underline{\varphi}_A)$. Then there exists $w_A\in(\Scal_n)^\Sigma$ such that $\kbold=w_A(\hbold)$, and the linearised Frobenius $\Phi_A$ on $\Dcris(r_A)$ admits $\underline{\varphi}_A$ as a set of eigenvalues. Furthermore, the flag $\Fil_\bullet\DdR(r_A)$ relative to the basis of $\DdR(r_A)$ induced by the refinement $\underline{\varphi}_A$, seen as a point of $\prod_{\tau\in\Sigma}(G/B)(A)$, lies in the cell $\prod_{\tau\in\Sigma}(Bw_{A,\tau}B/B)(A)$.
\end{prop}

Applied to the algebra $A=k(x)[\varepsilon]/(\varepsilon^2)$ of infinitesimal numbers over the local residue field of $x$, Proposition \ref{propIntroLikeBellaicheChenevier} shows that the commutative diagram
\[
	\begin{tikzcd}
	T_{\Wcrtildew{w_\sat},(r,\underline{\varphi})} \arrow[r,hook,"\diff\iota_{\hbold,w_\sat}"] \arrow[d,hook,"\subseteq"] & T_{\overline{U_{x_\sat}\times_{\Wcal^n_L}\Wcal^n_{w,w_\sat,\hbold,L}},x_\sat} \arrow[d,hook,"\diff\jmathauto"] \\
	T_{\Wcrtildew{w},(r,\underline{\varphi})} \arrow[r,hook,"\diff\iota_{\hbold,w}"] & T_{\Xtri,x} \,.
\end{tikzcd}
\]
is a cartesian square. This implies that $\dim(T_1\inter T_2)=\dim{T_{\Wcrtildew{w_\sat},(r,\underline{\varphi})}}$, from which we finally get \eqref{eqDimensionIntro} without difficulty.

\subsubsection*{Notation} We conclude this introduction with the main notation of the paper.

We fix a prime number $p$ and a finite extension $K$ of $\Q_p$, equipped with the $p$-adic norm $\abs{\cdot}_K$ normalised so that $\abs{p}_K=p^{-[K:\Q_p]}$. We write $\G_K$ for the absolute Galois group of $K$ and $\Gamma_K\coloneqq\Gal(K(\mu_{p^\infty})/K)$. We also write $K_0$ for the maximal unramified extension of $\Q_p$ inside $K$ and $\varphi$ for the Frobenius endomorphism of $K_0$, normalised so that it induces $x\longmapsto x^p$ on the residue field. If $L$ is another finite extension of $\Q_p$, we say that $L$ splits $K$ when the set $\Hom_{\Q_p}(K,L)$ of homomorphisms of fields over $\Q_p$ has cardinality $[K:\Q_p]$; in that case we write $\Sigma\coloneqq\Hom_{\Q_p}(K,L)$.

We write $\ArQ$ for the category of local artinian $\Q_p$-algebras whose residue field is a finite extension of $\Q_p$ (\emph{i.e.\ }finite-dimensional local $\Q_p$-algebras). For $A\in\ArQ$, we write $\m_A$ for its maximal ideal.

For a ring $R$, we write $\Filcat_R$ for the category of filtered $R$-modules, where objects are $R$-modules equipped with a decreasing exhaustive separated filtration by $R$-submodules, and morphisms are compatible families of $R$-linear maps in each degree. For $A\in\ArQ$, we write $\Rep_A(\G_K)$ for the category of representations of $\G_K$ over $A$, where objects are free $A$-modules of finite type equipped with a continuous linear action of $\G_K$, and morphisms are $\G_K$-equivariant $A$-linear maps. We use the usual functors of Fontaine $V\mapsto\left(\DdR(V),\Fil^\bullet\DdR(V)\right)$ from $\Rep_A(\G_K)$ to $\Filcat_{K\tens_{\Q_p}A}$ and $V\mapsto(\Dcris(V),\varphi)$ from $\Rep_A(\G_K)$ to the category of $K_0$-vector spaces equipped with a $\varphi$-semilinear endomorphism (called Frobenius and still noted $\varphi$). For $V\in\Rep_A(\G_K)$, we write $\Phi\coloneqq\varphi^{[K_0:\Q_p]}$ for the linearised Frobenius, which is a linear endomorphism of $\Dcris(V)$ (note that depending on the context, $\Phi$ can also denote a root system).

If $A\in\ArQ$ is an algebra over its residue field $L\coloneqq A/\m_AA$ and $L$ splits $K$, for $V\in\Rep_A(\G_K)$ which is de Rham and $\tau\in\Sigma$, we call \emph{$\tau$-Hodge-Tate weights} the integers $h_\tau$ such that $\rk_A\left(\Fil^{-h_\tau}\DdR(V)/\Fil^{-h_\tau+1}\DdR(V)\right)\neq0$. We count them with multiplicities, so that there are exactly $\rk_A(V)$ $\tau$-Hodge-Tate weights. We say that $V$ is \emph{regular} when the $\tau$-Hodge-Tate weights are distinct for each $\tau\in\Sigma$, \emph{i.e.\ }all $\tau$-Hodge-Tate weights have multiplicity 1.

With the same hypotheses on $A$, given $\kbold=(k_\tau)_{\tau\in\Sigma}\in\Z^\Sigma$, we write $z^\kbold$ for the character $K^\times\rightarrow A^\times,\,z\mapsto\prod_{\tau\in\Sigma}\tau(z)^{k_\tau}$. For $a\in A^\times$, we write $\nr(a)$ for the character $K^\times\rightarrow A^\times,\,z\mapsto a^{v_K(z)}$ where $v_K$ is the valuation on $K$ normalised to have image $\Z$ (for instance, $\nr(p^{[K_0:\Q_p]})(p)=p^{[K:\Q_p]}$; note that $v_K$ is not the valuation of the norm $\abs{\cdot}_K$). If $\delta\colon K^\times\to A^\times$ is any continuous character, then $\delta$ is locally $\Q_p$-analytic so it is a morphism of $p$-adic Lie groups; we write $\diff\delta\colon K\to A$ for its differential, which is a morphism of vector spaces over $\Q_p$. Note that the $\tau\colon K\inj L$ composed with $L\to A$, for $\tau\in\Sigma$, form a basis of $\Hom_{\Q_p\text{-vect. sp.}}(K,A)$ over $A$, which induces $\Hom_{\Q_p\text{-vect.\ sp.}}(K,A)\iso A^\Sigma$; under this isomorphism, we write $\wt_\tau(\delta)\in A$ for the weights of $\delta$ defined by $\diff\delta=(\wt_\tau(\delta))_{\tau\in\Sigma}$. If $\delta'$ is a continuous character $\intring_K^\times\to A^\times$, we write $\wt_\tau(\delta')\coloneqq\wt_\tau(\delta)$ where $\delta\colon K^\times\to A^\times$ is any character such that $\delta|_{\intring_K^\times}=\delta'$.

For any $A\in\ArQ$, we write $\Robba_{A,K}$ for the relative Robba ring over $A$ for $K$ (see \cite[Definition 2.2.2]{kpx}, where it is written $\Robba_A(\pi_K)$), which is equipped with a Frobenius endomorphism written $\varphi$ and a $\Gamma_K$-action. We write $\mathbf{Mod}^{\varphi,\Gamma_K}_A$ for the category of $(\varphi,\Gamma_K)$-modules over $A$; its objects are $\Robba_{A,K}$-modules $D$ which are finitely-generated over $\Robba_{A,K}$ and free over $\Robba_K$, equipped with semilinear actions of $\varphi$ and $\Gamma_K$ which commute and are continuous (as actions of free $\Robba_K$-modules), and such that $\Robba_K\varphi(D)=D$; its morphisms are $\Robba_{A,K}$-linear maps which commute with the respective actions of $\varphi$ and $\Gamma_K$. We use the usual functor $\Drig$ from $\Rep_A(\G_K)$ to $\mathbf{Mod}^{\varphi,\Gamma_K}_A$ (see \cite[Theorem 2.2.17]{kpx}). To a continuous character $\delta\colon\Q_p^\times\to A^\times$, we associate a $(\varphi,\Gamma_K)$-module $\Robba_{A,K}(\delta)$ of rank one over $\Robba_{A,K}$, defined in \cite[Construction 6.2.4]{kpx}.

\subsubsection*{Acknowledgements}

The results of this paper are part of the author's PhD work at the Laboratoire de Mathématiques d'Orsay. The author would like to thank his PhD advisor Christophe Breuil for his enlightening supervision, Ariane Mézard and Benjamin Schraen for valuable exchanges, as well as Gabriel Dospinescu and Ruochuan Liu for their careful review of the author's PhD dissertation. Finally, the author's debt to the work of Breuil-Hellmann-Schraen will be obvious to the reader.

\numberwithin{thm}{subsection}

\section{Crystalline trianguline representations}

We recall the notion of refinement of a crystalline representation $V$ over $A\in\ArQ$ in relation with that of triangulation of the associated $(\varphi,\Gamma)$-module. For $V$ regular, we characterise the refinement arising from a triangulation, relatively to the Hodge filtration, as a point of a Schubert cell of the flag variety of $\DdR(V)$.

\subsection{Refinements of a crystalline representation $V$ and flags on $\mathbf{D}_\mathrm{dR}(V)$}

We begin with a lemma about the functor of points of a Schubert cell of the flag variety.

\begin{dfn} \label{dfnFlagModules}
Let $R$ be a ring and let $M$ be a free $R$-module. A \emph{complete flag} on $M$ is an increasing sequence $\Fil_\bullet M$ of $R$-submodules of $M$
\[
	\Fil_0M\coloneqq\{0\}\subset\ldots\subset \Fil_iM\subset\ldots\subset \Fil_nM\coloneqq M
\]
with $n=\rk_R(M)$, such that $\Fil_iM/\Fil_{i-1}M$ is free of rank $1$ over $R$ for all $1\leq i\leq n$.
\end{dfn}

Fix an integer $n\in\Z_{>0}$, a ring $R$ and a free module $M$ of rank $n$ over $R$. Equip $M$ with a complete flag $\Fil_\bullet M=(\Fil_i(M))_{0\leq i\leq n}$, and fix a basis basis $(e_i)_{1\leq i\leq n}$ of $\Fil_\bullet M$, \emph{i.e.\ }a family of elements of $M$ such that $e_i\in\Fil_iM$ and $e_i\,\mathrm{mod}\,\Fil_{i-1}M$ generates $\Fil_iM/\Fil_{i-1}M$ for all $1\leq i\leq n$. Let $G\coloneqq\GL_{n,\Q_p}$ be the general linear group over $\Q_p$ and let $B$ be its Borel subgroup of upper triangular matrices; the Weyl group of $G$ is the group $\Scal_n$ of bijections of $\{1,\ldots,n\}$ to itself.

The map	$G(R)\to\{\text{complete flags of }M\}$ which sends a matrix $(a_{ij})_{1\leq i,j\leq n}\in G(R)$ to the flag $\left(\bigoplus_{j\leq k}Rb_j\right)_{1\leq k\leq n}$, where $b_j=\sum_{1\leq i\leq n}a_{ij}e_j$, factors through $B(R)$ and induces an bijection $G(R)/B(R)\isoto\{\text{complete flags of }M\}$, hence an injection
\begin{equation} \label{eqFlagsApoints}
	\{\text{complete flags of }M\} \inj (G/B)(R)
\end{equation}
(see \cite[\S5.4, \S5.5]{jantzen}). We call the image under \eqref{eqFlagsApoints} of a given flag the \emph{position of the flag relative to $\Fil_\bullet M$}.

Recall the Schubert decomposition on the flag variety $G/B$. We embed $\Scal_n$ in $G$ by mapping each $w\in\Scal_n$ to the matrix $\dot{w}=(w_{ij})_{1\leq i,j\leq n}$ defined by $w_{ij}=1$ if $i=w(j)$ and $w_{ij}=0$ if $i\neq w(j)$. Let $w\in\Scal_n$. The left-action of $B$ on $G$ by multiplication induces an action on $G/B$; write $BwB/B$ for the orbit of $\dot{w}B/B$ under this action. Then $BwB/B$ is a smooth locally closed subvariety of $G/B$ called a \emph{Schubert cell}, and we have the decomposition
\begin{equation} \label{eqStratification}
	G/B=\coprod_{w\in W}BwB/B
\end{equation}
(see \cite[Theorem 21.73]{milne}). The cells $BwB/B$ are stable under the left action of $B$, thus their Zariski closure $\overline{BwB/B}$ too, and this action is transitive on $BwB/B$. Thus, a closed cell $\overline{BwB/B}$ intersects an open cell $Bw'B/B$ if and only if $Bw'B\subseteq\overline{BwB/B}$, which happens if and only if $w'\preceq w$ for the Bruhat order (see \cite[Theorem 2.11]{schubertBGG}). Therefore \eqref{eqStratification} gives a good stratification of $G/B$, and we have a more general decomposition into locally closed strata
\[
	\overline{BwB/B}=\coprod_{w'\preceq w}Bw'B/B
\]
for all $w\in W$.

For flags, Schubert cells are elementarily described as such: given a flag $\Fcal$ of $M$, its position relative to $\Fil_\bullet M$ is in the closed cell $(\overline{BwB/B})(R)$ if and only if $\Fcal$ can be represented by a matrix $(a_{ij})_{1\leq i,j\leq n}\in G(R)$ such that $a_{ij}=0$ for all $i>w(j)$; and it is in the open cell $(BwB/B)(R)$ if and only if we can impose the additional condition $a_{w(j),j}\in R^\times$ for all $j$. (Thus we see that the Schubert decomposition induces a partition of the $R$-points if and only if $R$ is a field).

\begin{lem} \label{lemNonPathologicalFlag}
Let $R$ be a ring, let $M$ be a free $R$-module of rank $n\in\Z_{>0}$ equipped with a complete flag $\Fil_\bullet M$. Let $x\in(G/B)(R)$ be a point corresponding to a flag $(\Fcal_j)_{1\leq j\leq n}$ of $M$ under \eqref{eqFlagsApoints}. Then $x\in(BwB/B)(R)$ for some $w\in\Scal_n$ if and only if $(\Fcal_j\inter\Fil_{w(j)})/(\Fcal_j\inter\Fil_{w(j)-1})$ is free of rank $1$ over $R$ for all $1\leq j\leq n$.
\end{lem}

\begin{proof}
Let $(e_i)_{1\leq i\leq n}$ be a basis of $\Fil_\bullet M$.

If $x\in(BwB/B)(R)$, then there is a basis $(b_j)_{1\leq j\leq n}$ of $\Fcal_\bullet$ such that each $b_j$ can be written $b_j=\sum_{i\leq w(j)}r_{ij}e_i$ with $r_{w(j),j}\in R^\times$. Then, for all $1\leq j\leq n$, the injective morphism
\[
	(\Fcal_j\inter\Fil_{w(j)})/(\Fcal_j\inter\Fil_{w(j)-1}) \inj \Fil_{w(j)}/\Fil_{w(j)-1}
\]
sends $r_{w(j),j}^{-1}b_j\,\mathrm{mod}\,(\Fcal_j\inter\Fil_{w(j)-1})$ to $e_{w(j)}\,\mathrm{mod}\,\Fil_{w(j)-1}$, so it is an isomorphism.

Conversely, assume that $(\Fcal_j\inter\Fil_{w(j)})/(\Fcal_j\inter\Fil_{w(j)-1})$ is free of rank $1$ over $R$ for all $1\leq j\leq n$; it has a basis element that lifts to some $b_j\in\Fcal_j\inter\Fil_{w(j)}$. Then each $b_j$ can be written $b_j=\sum_{i\leq w(j)}r_{ij}e_i$ in $\Fcal_j$ with $r_{w(j),j}\in R^\times$. Therefore it is enough to prove that $(b_j)_{1\leq j\leq n}$ forms a basis of $\Fcal_\bullet$, \emph{i.e.\ }that $b_j\,\mathrm{mod}\,\Fcal_{j-1}$ is a basis of $\Fcal_j/\Fcal_{j-1}$ for all $1\leq j\leq n$. We proceed by induction on $j$. Let $b=\sum_{1\leq i\leq n}r_ie_i\in\Fcal_j$ such that $b\,\mathrm{mod}\,\Fcal_{j-1}$ is a basis of $\Fcal_j/\Fcal_{j-1}$. As $b_j\in\Fcal_j$, one can write
\begin{equation} \label{eqFlagBasis}
	b_j = \sum_{k<j}\lambda_kb_k + \lambda b
\end{equation}
for some $\lambda_k,\lambda\in R$. Let $I$ be the ideal of $R$ generated by $\lambda$. If $\lambda_k\notin I$ for some $k<j$, take such a $\lambda_k$ with $w(k)$ maximal. Then taking the coordinate along $e_{w(k)}$ of \eqref{eqFlagBasis} gives
\[
	r_{w(k),j} = \lambda_kr_{w(k),k} + \sum_{k'<j\text{ s.t.\ }w(k')>w(k)}\lambda_{k'}r_{w(k),k'} + \lambda r_{w(k)} \notin I
\]
since $r_{w(k),k}\in R^\times$. In particular $r_{w(k),j}\neq 0$, so $w(k)\leq w(j)$. This proves that $\lambda_k\in I$ for all $k<j$ such that $w(k)>w(j)$. Then \eqref{eqFlagBasis} along $e_{w(j)}$ gives
\[
	r_{w(j),j} = \sum_{k<j\text{ s.t.\ }w(k)>w(j)}\lambda_{k}r_{w(j),k} + \lambda r_{w(j)} \in I \,.
\]
As $r_{w(j),j}\in R^\times$, we deduce $I=R$. Hence $\lambda\in R^\times$, so $b_j\,\mathrm{mod}\,\Fcal_{j-1}$ is a basis of $\Fcal_j/\Fcal_{j-1}$; this concludes the proof.
\end{proof}

We are mainly concerned with the situation of $R=K\otimes_{\Q_p}A$ and $M=\DdR(V)$, where we fix an algebra $A\in\ArQ$ and a de Rham representation $V\in\Rep_A(\G_K)$. More specifically, we make the following two assumptions: (i) $A$ is an algebra over a finite extension $L$ of $\Q_p$ which is the residue field of $A$ and which splits $K$, and (ii) $V$ is regular.

Then we define a complete flag $\Fil_\bullet\DdR(V)$ on the free $(K\tens_{\Q_p}A)$-module $\DdR(V)$ by $\Fil_0\DdR(V)\coloneqq0$ and
\[
	\Fil_i\DdR(V) \coloneqq \bigoplus_{\tau\in\Sigma}\Fil^{-h_{\tau,i}}\DdR(V)_\tau ,\quad 1\leq i\leq n
\]
where $n=\rk_A(V)$ and $h_{\tau,1}<\ldots<h_{\tau,n}$ are the $\tau$-Hodge-Tate weights of $V$.

Next, we further assume that $V$ is crystalline. Recall the following definition from \cite[\S2.4]{bellaicheChenevier}.

\begin{dfn} \label{dfnRefinement}
Let $A\in\ArQ$, let $V\in\Rep_A(\G_K)$ be any crystalline representation. A \emph{refinement} of $V$ is a complete flag on $\Dcris(V)$ which is stable by the linearised Frobenius $\Phi$.
\end{dfn}

A refinement $\Fcal_\bullet$ of $V$ induces a complete flag $K\tens_{K_0}\Fcal_\bullet$ on $\DdR(V)$, and the position of $\Fil_\bullet\DdR(V)$ relative to $K\tens_{K_0}\Fcal_\bullet$ determines an $A$-point of
\[
	\left(\Res_{K/\Q_p}(G/B)_K\right)\times_{\Q_p}L \iso \prod_{\tau\in\Sigma}(G/B)_L
\]
where $G=\GL_{n,\Q_p}$ and $B$ is the Borel subgroup of upper triangular matrices.

\begin{rem} \label{rmkRefinementToEigenvalues}
A refinement of $V$ determines an ordered sequence of eigenvalues of $\Phi$, \emph{i.e.\ }a sequence $(\varphi_1,\ldots,\varphi_n)\in(K_0\tens_{\Q_p}A)^n$ such that the characteristic polynomial of $\Phi$ can be written $P_\Phi(\lambda)=\prod_{i=1}^n(\lambda-\varphi_i)\in(K_0\tens_{\Q_p}A)[\lambda]$. The following proposition is a partial converse of this fact.
\end{rem}

\begin{prop} \label{propEigenvaluesToRefinement}
Let $\varphi_1,\ldots,\varphi_n\in A$ be such that $(1\otimes\varphi_i)_{1\leq i\leq n}$ is a sequence of eigenvalues of the Frobenius $\Phi$ on $\Dcris(V)$ in the sense of Remark \ref{rmkRefinementToEigenvalues}. Assume that the $\varphi_i$ have pairwise distinct reductions modulo $\m_A$. Then there is a unique refinement of $V$ associated to $(1\otimes\varphi_i)_{1\leq i\leq n}$.
\end{prop}

The proof uses the following classic result of linear algebra.

\begin{lem} \label{lemEigenvalues}
Let $R$ be a commutative ring with unity, let $M$ be an $R$-module and $f\in\End_R(M)$.
\begin{enumerate}[label=\emph{(\arabic*)},ref=(\arabic*)]
\item \label{itemEigendecomposition}
Assume that there are $\lambda_1,\ldots,\lambda_m\in R$ for some $m\in\Z_{>0}$, with $\lambda_i-\lambda_j\in R^\times$ for all $1\leq i<j\leq m$, such that $(f-\lambda_1)\cdots(f-\lambda_m) = 0 \in \End_R(M)$. Then $M$ is the direct sum of eigenspaces
\begin{equation} \label{eqDirectSumEigenspaces}
	M = \ker(f-\lambda_1) \oplus \cdots \oplus \ker(f-\lambda_m) .
\end{equation}
\item \label{itemEigenbasis}
Assume that $R$ is local with maximal ideal $\m$, and that $M$ is free of rank $n\in\Z_{>0}$. Also assume that the characteristic polynomial of $f$ can be written $P_f(\lambda)=\prod_{i=1}^n(\lambda-\lambda_i)\in R[\lambda]$ with the $\lambda_i$ having pairwise distinct reductions $\overline{\lambda_i}$ modulo $\m$. Then the eigenspaces $\ker(f-\lambda_i)$, $1\leq i\leq n$, are free of rank $1$ over $R$ and any choice of generators of these eigenspaces form a basis of $M$.
\end{enumerate}
\end{lem}

\begin{proof}
We first prove \ref{itemEigendecomposition}. One can find polynomials $P_1,\ldots,P_m\in R[\lambda]$ such that
\begin{equation} \label{eqBezoutPolynomials}
	\sum_{1\leq i\leq m}P_i(\lambda)\prod_{j\neq i}(\lambda-\lambda_j) = 1 \in R[\lambda] \,;
\end{equation}
we show this by induction on $m$. The case $m=1$ is trivial with $P_1=1$; for $m>1$, let $Q_1,\ldots,Q_{m-1}\in R[\lambda]$ satisfy
\[
	\sum_{1\leq i\leq m-1}Q_i(\lambda)\prod_{j\neq i,m}(\lambda-\lambda_j) = 1 \in R[\lambda] \,.
\]
For $l\coloneqq\prod_{1\leq i<m}(\lambda_m-\lambda_i)\in R^\times$, there is $Q\in R[\lambda]$ such that $\prod_{1\leq i<m}(\lambda-\lambda_i)=Q(\lambda)(\lambda-\lambda_m)+l$. Then taking $P_i\coloneqq-l^{-1}Q_iQ\in R[\lambda]$ for $1\leq i<m$ and $P_m\coloneqq l^{-1}$ gives \eqref{eqBezoutPolynomials}.

Applying \eqref{eqBezoutPolynomials} to $\lambda=f$, we see that any $v\in M$ decomposes as $v=v_1+\ldots+v_m$, where $v_i\coloneqq\prod_{j\neq i}(f-\lambda_j)P_i(f)v$ for $1\leq i\leq m$, and obviously $v_i\in\ker(f-\lambda_i)$ by the assumption. Now to check that the sum $M=\ker(f-\lambda_1)+\ldots+\ker(f-\lambda_m)$ is direct, let $v_1+\ldots+v_m=0$, with $v_i\in\ker(f-\lambda_i)$. We show by induction on the number of nonzero $v_i$ that $v_i=0$ for all $i$. Take $i_0$ such that $v_{i_0}\neq0$; applying $f-\lambda_{i_0}$ gives $\sum_{i\neq i_0}(\lambda_i-\lambda_{i_0})v_i=0$, so by induction $(\lambda_i-\lambda_{i_0})v_i=0$ for all $1\leq i\leq m$. We get $v_i=0$ for $i\neq i_0$ since $\lambda_i-\lambda_{i_0}$ is a unit, then $v_{i_0}=0$. This proves \ref{itemEigendecomposition}.

We now prove \ref{itemEigenbasis}. The reduction of $f$ to the $(R/\m)$-vector space $M/\m M$ has pairwise distinct eigenvalues $\overline{\lambda_i}$; choose a corresponding diagonalisation basis $(\overline{v_i})_{1\leq i\leq n}$. Choose a lift $w_i\in M$ of each $\overline{v_i}$. By \ref{itemEigendecomposition} and Cayley-Hamilton's theorem, we can write $w_i=v_{i,1}+\ldots+v_{i,n}$ with $v_{i,j}\in\ker(f-\lambda_j)$ for all $1\leq i,j\leq n$. Then, in $M/\m M$,
\[
	\sum_{1\leq j\leq n}(\overline{\lambda_j}-\overline{\lambda_i})\overline{v_{i,j}} = f(\overline{v_i})-\overline{\lambda_i}\overline{v_i} = 0
\]
for all $1\leq i\leq n$. Applying \ref{itemEigendecomposition} to $M/\m M$, we see that the eigenvectors $(\overline{\lambda_j}-\overline{\lambda_i})\overline{v_{i,j}}$ are zero for all $1\leq i,j\leq n$, which means, since the $\overline{\lambda_i}$ are pairwise distinct, that $v_{i,j}\in\m$ for all $i\neq j$. Therefore $v_i\coloneqq v_{i,i}=w_i-\sum_{j\neq i}v_{i,j}$ lifts $\overline{v_i}$. By Nakayama's lemma, the $v_i$, $1\leq i\leq n$, generate $M$. Since $M$ is free of rank $n$, the $v_i$ actually form a basis of $M$. Since $Rv_i\subseteq\ker(f-\lambda_i)$ for each $i$, \eqref{eqDirectSumEigenspaces} gives $Rv_i=\ker(f-\lambda_i)$. This proves \ref{itemEigenbasis}.
\end{proof}

\begin{proof}[Proof of Proposition \ref{propEigenvaluesToRefinement}]
Write $K_0\tens_{\Q_p}A=\bigoplus_{\tau\in\Sigma_0}Ae_\tau$, where $e_\tau^2=e_\tau$ for $\tau\in\Sigma_0\coloneqq\Hom_{\Q_p}(K_0,L)$. Then $\Dcris(V)=\bigoplus_{\tau\in\Sigma_0}\Dcris(V)_\tau$ where each $\Dcris(V)_\tau\coloneqq e_\tau\Dcris(V)$ is a free $A$-module of rank $n$. Each $\Dcris(V)_\tau$ is stabilised by $\Phi$, and the induced $\Phi_\tau$ is an endomorphism of $A$-modules with sequence of eigenvalues $(\varphi_i)_{1\leq i\leq n}$. Applying Lemma \ref{lemEigenvalues}\ref{itemEigenbasis}, we get a diagonalisation basis $(b_{\tau,i})_{1\leq i\leq n}$ of $\Dcris(V)_\tau$ for $(\varphi_i)_{1\leq i\leq n}$. Setting $b_i\coloneqq\sum_{\tau\in\Sigma_0}b_{\tau,i}$, the family $(b_i)_{1\leq i\leq n}$ is a diagonalisation basis of $\Dcris(V)$ for $(1\tens\varphi_i)_{1\leq i\leq n}$.

Let $\Fcal_\bullet$ be any refinement of $V$ associated to $(1\otimes\varphi_i)_{1\leq i\leq n}$. Let $(a_i)_{1\leq i\leq n}$ be a basis of $\Fcal_\bullet$, and write $a_j=\sum_{j=1}^na_{ij}b_i$ for all $1\leq j\leq n$. Using an induction on the columns of the matrix $(a_{ij})_{1\leq i,j\leq n}$, we see that it is upper triangular, hence has invertible diagonal entries. Therefore $(b_i)_{1\leq i\leq n}$ is a basis of $\Fcal_\bullet$; this proves unicity of $\Fcal_\bullet$. Conversely, the flag arising from $(b_i)_{1\leq i\leq n}$ induces the ordered sequence of eigenvalues $(1\otimes\varphi_i)_{1\leq i\leq n}$; this proves existence.
\end{proof}

\subsection{Triangulations and refinements}

We recall the notion of triangular $(\varphi,\Gamma)$-module from \cite[Definition 2.3.2]{bellaicheChenevier} and \cite[Definition 3.3.8]{bhs3} in the following definition.

\begin{dfn}
Let $A\in\ArQ$. Let $D$ be a $(\varphi,\Gamma_K)$-module over $\Robba_{A,K}$ which is free as an $\Robba_{A,K}$-module. A \emph{triangulation} of $D$ is a strictly increasing sequence $\Fil_\bullet D$ of $(\varphi,\Gamma_K)$-submodules over $\Robba_{A,K}$
\[
	\Fil_0D\coloneqq\{0\} \subsetneq\ldots\subsetneq \Fil_iD \subsetneq\ldots\subsetneq \Fil_dD\coloneqq D
\]
with $d=\rk_{\Robba_{A,K}}(D)$, such that each $\Fil_i(D)$ is a direct summand of $D$ as $\Robba_{A,K}$-modules. We say that $D$ is \emph{triangulable} if it can be equipped with a triangulation, and the data of $D$ with a triangulation is called a \emph{triangular} $(\varphi,\Gamma_K)$-module. A representation $V\in\Rep_A(\G_K)$ is called \emph{trianguline} if $\Drig(V)$ is triangulable.
\end{dfn}

Given a triangular $(\varphi,\Gamma_K)$-module $\left(D,\Fil_\bullet D\right)$ over $\Robba_{A,K}$, the graded pieces
\[
	\gr_iD \coloneqq \Fil_iD/\Fil_{i-1}D \,,\quad 1\leq i\leq d
\]
of the triangulation are free of rank 1 as $\Robba_{A,K}$-modules. When $A$ is a $K_0$-algebra, according to Proposition \ref{propRank1PhiGamma} below, for each $1\leq i\leq d$ there is a unique continuous character $\delta_i\colon K^\times\to A^\times$ such that $\gr_iD\iso\Robba_{A,K}(\delta_i)$.

\begin{prop}[Colmez, Kedlaya-Pottharst-Xiao] \label{propRank1PhiGamma}
Let $A\in\ArQ$ be a $K_0$-algebra. Given any $(\varphi,\Gamma_K)$-module $D$ over $\Robba_{A,K}$ which is free of rank $1$ as $\Robba_{A,K}$-module, there is a unique continuous character $\delta\colon K^\times\to A^\times$ such that $D\iso\Robba_{A,K}(\delta)$ as $(\varphi,\Gamma_K)$-modules.
\end{prop}

\begin{proof}
This is \cite[Lemma 6.2.13]{kpx}; see also \cite[Proposition 3.1]{colmezTriangulines} over $\Robba_{L,\Q_p}$ and \cite[Proposition 2.3.1]{bellaicheChenevier} over $\Robba_{A,\Q_p}$.
\end{proof}

\begin{dfn}
Let $A\in\ArQ$ be a $K_0$-algebra. Let $\left(D,\Fil_\bullet D\right)$ be a triangular $(\varphi,\Gamma_K)$-module of rank $d$ over $\Robba_{A,K}$. The \emph{parameter} of the triangulation $\Fil_\bullet D$ is the unique continuous character
\[
	\underline{\delta}=(\delta_1,\ldots,\delta_d) \colon (K^\times)^d \to A^\times
\]
such that $\gr_iD\iso\Robba_{A,K}(\delta_i)$ for all $1\leq i\leq d$.
\end{dfn}

The following result somewhat generalises \cite[Proposition 2.4.1]{bellaicheChenevier}.

\begin{prop} \label{propLikeBellaicheChenevier}
Let $A\in\ArQ$ be an algebra over its residue field $L\coloneqq A/\m_AA$ such that $L$ splits $K$. Let $V\in\Rep_A(\G_K)$ be a regular crystalline representation of rank $n\in\Z_{>0}$ and of Hodge-Tate weights $(h_{\tau,1}<\ldots<h_{\tau,n})_{\tau\in\Sigma}$.

Assume that there exist $\kbold=(k_{\tau,i})_{\tau\in\Sigma,1\leq i\leq n}\in(\Z^n)^\Sigma$ and $\varphi_i\in A^\times$ for $1\leq i\leq n$ with pairwise distinct reductions in $L$, such that $V$ is trianguline of parameter
\[
	\underline{\delta} = z^\kbold\nr(\underline{\varphi}) \coloneqq \left(z^{\kbold_1}\nr(\varphi_1),\ldots,z^{\kbold_n}\nr(\varphi_n)\right) .
\]
Then:
\begin{enumerate}[label=\emph{(\arabic*)},ref=(\arabic*)]
\item \label{itemHTweightsTriangulation}
There exists $w=(w_\tau)_{\tau\in\Sigma_n}\in(\Scal_n)^\Sigma$ such that $k_{\tau,i}=h_{\tau,w_\tau^{-1}(i)}$ for all $\tau\in\Sigma$ and $1\leq i\leq n$.
\item \label{itemRefinementTriangulation}
The linearised Frobenius $\Phi$ on $\Dcris(V)$ admits $(1\tens\varphi_i)_{1\leq i\leq n}$ as eigenvalues, and this ordering determines a refinement $\Fcal_\bullet$ of $V$.
\item \label{itemFlagTriangulation}
The position $x\in\prod_{\tau\in\Sigma}(G/B)(A)$ of $\Fil_\bullet\DdR(V)$ relative to $\Fcal_{\dR,\bullet}\coloneqq K\tens_{K_0}\Fcal_\bullet$ lies in the cell $\prod_{\tau\in\Sigma}(Bw_\tau B/B)(A)$.
\end{enumerate}
\end{prop}

\begin{proof}
Recall from Berger's work the functor $\Dcalcris$ from $\mathbf{Mod}^{\varphi,\Gamma_K}_A$ to the category of $K_0$-modules equipped with a $\varphi$-semilinear endomorphism (see \cite[Théorème 3.6]{berger02}) and the functor $\DcaldR$ from $\mathbf{Mod}^{\varphi,\Gamma_K}_A$ to $\Filcat_{K\tens_{\Q_p}A}$ (see \cite[Théorème A]{berger08}). They are characterised by the property that $\Dcalcris(\Drig(V))=\Dcris(V)$ and $\DcaldR(\Drig(V))=\DdR(V)$ for all $V\in\Rep_A(\G_K)$. It is also known that $\Dcalcris$ and $\DcaldR$ are exact functors on the full subcategory of crystalline $(\varphi,\Gamma_K)$-modules, \emph{i.e.\ }$(\varphi,\Gamma_K)$-modules $D$ such that $\dim_{K_0}\Dcalcris(D)=\rk_{\Robba_K}D$.

Let $(D_i)_{0\leq i\leq n}$ be a triangulation of $\Drig(V)$ of parameter $\underline{\delta}$. For $0\leq i\leq n$, write $\Fcal_i\coloneqq\Dcalcris(D_i)$. The $D_i$ are crystalline (as subobjects of $\Drig(V)$), hence by exactness of $\Dcalcris$ each exact sequence
\[
	0 \to D_{i-1} \to D_i \to \Robba_{A,K}\left(z^{\kbold_i}\nr(\varphi_i)\right) \to 0
\]
induces an exact sequence
\[
	0 \to \Fcal_{i-1} \to \Fcal_i \to \Dcalcris\left(\Robba_{A,K}\left(z^{\kbold_i}\nr(\varphi_i)\right)\right) \to 0
\]
of $K_0\tens_{\Q_p}A[\varphi]$-modules. It is known that $\Dcalcris\left(\Robba_{A,K}\left(z^{\kbold_i}\nr(\varphi_i)\right)\right)$ is free of rank one over $K_0\tens_{\Q_p}A$ and that its linearised Frobenius acts by multiplication by $\varphi_i$ \cite[Example 6.2.6]{kpx}. Therefore $\Phi$ acts on $\Fcal_i/\Fcal_{i-1}$ by multiplication by $1\tens\varphi_i$, thus $(1\tens\varphi_i)_{1\leq i\leq n}$ is a sequence of eigenvalues of $\Phi$ and $(\Fcal_i)$ is the unique (see Proposition \ref{propEigenvaluesToRefinement}) refinement of $V$ associated to $(1\tens\varphi_i)_{1\leq i\leq n}$. This proves \ref{itemRefinementTriangulation}.

Similarly, the exact functor $\DcaldR$ induces for each $1\leq i\leq n$ an exact sequence
\[
	0 \to \Fcal_{\dR,i-1} \to \Fcal_{\dR,i} \to \DcaldR\left(\Robba_{A,K}\left(z^{\kbold_i}\nr(\varphi_i)\right)\right) \to 0
\]
in $\Filcat_{K\tens_{\Q_p}A}$. It induces, for each $1\leq i\leq n$, each $\tau\in\Sigma$ and each $k\in\Z$, an isomorphism
\begin{multline} \label{eqRefinementHT1}
	\Fil^k\left(\Fcal_{\dR,i}\right)_\tau/\Fil^k\left(\Fcal_{\dR,i-1}\right)_\tau \\
	\isoto \Fil^k\DcaldR\left(\Robba_{A,K}\left(z^{\kbold_i}\nr(\varphi_i)\right)\right)_\tau \iso
	\begin{cases}
		0 &\text{if } k\neq k_{\tau,i} \\
		A &\text{if } k=k_{\tau,i}
	\end{cases}
\end{multline}
(see \cite[Example 6.2.6]{kpx} for the last isomorphism). Applying \eqref{eqRefinementHT1} to $k=k_{\tau,i}+1$, we get $\Fil^{k_{\tau,i}+1}\left(\Fcal_{\dR,i}\right)_\tau/\Fil^{k_{\tau,i}+1}\left(\Fcal_{\dR,i-1}\right)_\tau\isoto0$, thus
\begin{multline} \label{eqRefinementHT2}
	\Fil^{k_{\tau,i}}\left(\Fcal_{\dR,i}\right)_\tau/\Fil^{k_{\tau,i}+1}\left(\Fcal_{\dR,i}\right)_\tau = \Fil^{k_{\tau,i}}\left(\Fcal_{\dR,i}\right)_\tau/\Fil^{k_{\tau,i}+1}\left(\Fcal_{\dR,i-1}\right)_\tau \\
	\surj \Fil^{k_{\tau,i}}\left(\Fcal_{\dR,i}\right)_\tau/\Fil^{k_{\tau,i}}\left(\Fcal_{\dR,i-1}\right)_\tau
\end{multline}
The exactness of $\DdR$ also shows that $\Fil^k\Fcal_{\dR,i}=\Fcal_{\dR,i}\inter\Fil^k\DdR(V)$ for all $k\in\Z$, therefore \eqref{eqRefinementHT2} rephrases as
\begin{multline} \label{eqRefinementHT3}
	\left(\Fcal_{\dR,i}\inter\Fil^{k_{\tau,i}}\DdR(V)\right)_\tau / \left(\Fcal_{\dR,i}\inter\Fil^{k_{\tau,i}+1}\DdR(V)\right)_\tau \\
	\surj \Fil^{k_{\tau,i}}\left(\Fcal_{\dR,i}\right)_\tau/\Fil^{k_{\tau,i}}\left(\Fcal_{\dR,i-1}\right)_\tau .
\end{multline}
On the other hand, there is an obvious injective morphism
\begin{multline} \label{eqRefinementHT4}
	\left(\Fcal_{\dR,i}\inter\Fil^{k_{\tau,i}}\DdR(V)\right)_\tau / \left(\Fcal_{\dR,i}\inter\Fil^{k_{\tau,i}+1}\DdR(V)\right)_\tau \\
	\inj \Fil^{k_{\tau,i}}\DdR(V)_\tau/\Fil^{k_{\tau,i}+1}\DdR(V)_\tau .
\end{multline}
The right-hand side of \eqref{eqRefinementHT3} is free of rank one over $A$ by \eqref{eqRefinementHT1} applied to $k=k_{\tau,i}$. Therefore the right-hand side of \eqref{eqRefinementHT4} is nonzero: this means that the jumps of the Hodge filtration of $V$ are in degrees $(k_{\tau,i})_{\tau\in\Sigma,1\leq i\leq n}$, which proves \ref{itemHTweightsTriangulation}.

Furthermore, by regularity of $V$, the right-hand side of \eqref{eqRefinementHT4} is actually free of rank one over $A$. Comparing dimensions over $L$, we deduce that \eqref{eqRefinementHT3} and \eqref{eqRefinementHT4} are isomorphisms, hence the $A$-module $\left(\Fcal_{\dR,i}\inter\Fil^{k_{\tau,i}}\DdR(V)\right)_\tau / \left(\Fcal_{\dR,i}\inter\Fil^{k_{\tau,i}+1}\DdR(V)\right)_\tau$ is free of rank $1$. Then \ref{itemFlagTriangulation} follows from Lemma \ref{lemNonPathologicalFlag} (note that our computations give the position of $\Fcal_{\dR,\bullet}$ relative to $\Fil_\bullet\DdR(V)$, which is inverse of the position of the position of $\Fil_\bullet\DdR(V)$ relative to $\Fcal_{\dR,\bullet}$).
\end{proof}

\section{The trianguline variety and companion points on the patched eigenvariety} \label{secVarieteTrianguline}

In this section, we recall the definition of the trianguline variety as well as some of the results of \cite{bhs1} and \cite{bhs2}.

Let $L$ be a finite extension of $\Q_p$ which splits $K$, with integer ring $\intring_L$ and residue field $k_L$; recall the notation $\Sigma\coloneqq\Hom_{\Q_p}(K,L)$. We fix $n\in\Z_{>0}$ and a continuous representation
\[
	\bar{r} \colon \G_K \to \GL_n(k_L)
\]
(where $k_L$ has the discrete topology).

\subsection{The trianguline variety}

Let $\Rr$ be the \emph{framed deformation ring of $\bar{r}$} from the work of Mazur \cite{mazur} and Kisin \cite{kisin09}; it is a complete local noetherian $\intring_L$-algebra with residue field $k_L$, which pro-represents the functor sending a local artinian $\intring_L$-algebra $A$ with residue field $k_L$ to the set of framed deformations of $\bar{r}$ over $A$. Define the \emph{framed deformation variety of $\bar{r}$} as the rigid $L$-analytic space $\Xfrakr\coloneqq\Spf(\Rr)^\rig$, where $\Spf$ is the formal spectrum and $(-)^\rig$ is the rigidification functor of Berthelot \cite[\S0.2]{berthelot} (see also \cite[\S7.1]{dejong}). Using \cite[Proposition 7.1.7]{dejong} and \cite[Proposition 2.3.5]{kisin09}, we see that for any local artinian $L$-algebra $A$ with residue field denoted by $E$, the set of $A$-points $\Xfrakr(A)$ is the set of representations $r_A\colon\G_K\to\GL_n(A)$ such that the composition of $r_A$ with $\GL_n(A)\surj\GL_n(E)$ has image in $\GL_n(\intring_E)$ and is a framed deformation of the composition of $\bar{r}$ with $\GL_n(k_L)\to\GL_n(k_E)$ (where $\intring_E$ and $k_E$ are respectively the integer ring and residue field of $E$).

Let $\Wcal\coloneqq\widehat{\intring_K^\times}$ be the rigid analytic space over $\Q_p$ that parametrises continuous characters of $\intring_K^\times$, \emph{i.e.\ }such that $\Wcal(X)=\Hom_{\mathrm{cont}}\left(\intring_K^\times,\Gamma(X,\func_X)^\times\right)$ for all rigid analytic space $X$ over $\Q_p$; the existence and smoothness of $\Wcal$ can be seen using a decomposition $\intring_K^\times\iso H\times\Z_p^{[K:\Q_p]}$ of topological groups, where $H$ is a finite torsion abelian group (see \cite[Proposition II.5.7]{neukirch}). Similarly, let $\Tcal\coloneqq\widehat{K^\times}$ be the rigid analytic space over $\Q_p$ that parametrises continuous characters of $K^\times$; it is also smooth and any decomposition $K^\times\iso\Z\times\intring_K^\times$ induces an isomorphism $\Tcal\iso\Gm^\rig\times_{\Q_p}\Wcal$ of rigid analytic spaces over $\Q_p$. The restriction of a $K^\times$-character to $\intring_K^\times$ induces a morphism $\Tcal\to\Wcal$ which corresponds to the projection $\Gm^\rig\times_{\Q_p}\Wcal\surj\Wcal$.

Write $\Tcal_L\coloneqq\Tcal\times_{\Q_p}L$. The set
\[
	\enstq{(\delta_i)_{1\leq i\leq n}\in\Tcal_L^n(L)}{\begin{array}{l}
		\delta_i\delta_j^{-1}=z^{-\kbold}\text{ or }\delta_i\delta_j^{-1}=N_{K/\Q_p}(z)z^\kbold\abs{z}_K \\
		\text{for some }i\neq j, \kbold\in\Z_{\geq0}^\Sigma
	\end{array}}
\]
is an analytic subset of $\Tcal_L^n$ (in the sense of \cite[\S9.5.2]{bgr}), so its complement is a Zariski-open subset of $\Tcal_L^n$, which we call $\Tnreg$.

\begin{dfn}[Hellmann] \label{dfnTriangulineVariety}
Let $\Utri$ be the subset
\[
	\Utri\coloneqq\enstq{(r,\underline{\delta})\in\Xfrakr\times_L\Tnreg}{r\text{ is trianguline of parameter }\underline{\delta}}
\]
of $\Xfrakr\times_L\Tnreg$. The \emph{trianguline variety} is the Zariski-closure $\Xtri$ of $\Utri$ in $\Xfrakr\times_L\Tcal_L^n$, with its structure of reduced rigid analytic space over $L$. The composition of the inclusion $\Xtri\inj\Xfrakr\times_L\Tcal_L^n$, the projection $\Xfrakr\times_L\Tcal_L^n\surj\Tcal_L^n$ and the restriction morphism $\Tcal_L^n\to\Wcal_L^n$ is a morphism of rigid analytic varieties over $L$
called the \emph{weight morphism} and denoted $\omega\colon\Xtri\to\Wcal_L^n$.
\end{dfn}

Note that $\Utri$ (resp.\ $\Xtri$) is also noted $U_\tri^\square(\bar{r})^\reg$ or $U_\tri(\bar{r})$ (resp.\ $X_\tri(\bar{r})$) in the literature.

\begin{thm}[Breuil-Hellmann-Schraen] \label{thmXtriGeometry}
\begin{enumerate}[label=\emph{(\arabic*)},ref=(\arabic*)]
\item
The rigid analytic space $\Xtri$ is equidimensional of dimension $n^2+[K:\Q_p]\frac{n(n+1)}{2}$.
\item
The subset $\Utri$ of $\Xtri$ is Zariski-open and Zariski-dense.
\item \label{itemUtriSmooth}
The open set $\Utri$ is smooth over $\Q_p$, and the restriction of $\omega$ to $\Utri$ is a smooth morphism.
\end{enumerate}
\end{thm}

\begin{proof}
This is \cite[Théorème 2.6]{bhs1}.
\end{proof}

\begin{dfn}
A point $x=(r,\underline{\delta})\in\Xtri$ is said to be \emph{crystalline} if the representation $r\colon\G_K\to\GL_n(k(x))$ is crystalline ($k(x)$ is the local residue field of $\Xtri$ at $x$).
\end{dfn}

\begin{lem}[Breuil-Hellmann-Schraen] \label{lemXtriPoints}
Let $x=(r,\underline{\delta})\in\Xtri$ be a crystalline point. Then there exist $\hbold=(h_{\tau,i})\in(\Z^\Sigma)^n$ and $\underline{\varphi}=(\varphi_1,\ldots,\varphi_n)\in (k(x)^\times)^n$ such that
\[
	\underline{\delta} = z^\hbold\nr(\underline{\varphi}) \coloneqq \left(z^{\hbold_1}\nr(\varphi_1),\ldots,z^{\hbold_n}\nr(\varphi_n)\right) \,.
\]
Moreover, the $1\tens\varphi_i\in K_0\tens_{\Q_p}k(x)$ are the eigenvalues of the linearised Frobenius $\Phi$ on $\Dcris(r)$ and, for each $\tau\in\Sigma$, the multiset of $\tau$-Hodge-Tate weights of $r$ is $\enstq{h_{\tau,i}}{1\leq i\leq n}$.

If $\varphi_i\varphi_j^{-1}\notin\{1,p^{[K_0:\Q_p]}\}$ for $i\neq j$, then there also exists $\hbold'\in(\Z^\Sigma)^n$ such that $(r,z^{\hbold'}\nr(\underline{\varphi}))\in\Utri$.
\end{lem}

\begin{proof}
This is \cite[Lemma 2.1]{bhs2}; the second part is a consequence of its proof. Note that necessarily $\enstq{h_{\tau,i}}{1\leq i\leq n}=\enstq{h'_{\tau,i}}{1\leq i\leq n}$ for all $\tau\in\Sigma$.
\end{proof}

\subsection{The patched eigenvariety}

For the rest of the paper, we make the assumption $p\nmid2n$ (which implies in particular $p>2$).

We consider the following global setting. Let $F$ be a CM field, \emph{i.e.\ }a quadratic totally imaginary extension of a totally real number field $F^+$, and write $S_p$ for the set of places of $F^+$ dividing $p$. We assume that $F/F^+$ is unramified in all finite places, and that all places $v\in S_p$ split in $F$ and satisfy $F^+_v\iso K$. Let $G$ be a unitary group in $n$ variables over $F^+$ quasi-split in all finite places of $F^+$, such that $G\times_{F^+}F\iso\GL_{n,F}$ and $G(F\tens_\Q\mathbb{R})$ is compact. Let $U^p=\prod_{v\notin S_p\text{ finite}}U_v\subset G(\mathbb{A}^{p\infty}_{F^+})$ be a tame level, where each $U_v$ is a compact subgroup of $F^+_v$ which is assumed to be hyperspecial whenever $v$ is inert. We then fix a finite set $S$ of finite places of $F^+$ containing $S_p$ and all the places $v$ such that $U_v$ is not hyperspecial; in particular all places $v\in S$ split in $F$. Let $\bar{\rho}\colon\G_{F^+}\to\GL_n(k_L)$ be a continuous irreducible representation which is automorphic (see for example \cite[Definition 5.3.1]{emertonGee14}). In particular, it follows from \cite[\S5.1, \S5.2, \S5.3]{emertonGee14} or \cite[\S2.1, \S2.3]{caraianiEtAl} that $\bar{\rho}$ \og{}comes from\fg{} $p$-adic automorphic forms on $G(\mathbb{A}_{F^+})$ of tame level $U^p$; see \cite[\S2.4]{bhs1} for a detailed exposition (see also the beginning of \cite[\S3.1, \S3.2]{bhs2}). Also, $\bar{\rho}\left(\Gal\left(\overline{F}/F(\zeta_p)\right)\right)$ is adequate in the sense of \cite[Def.\ 2.3]{thorne}. Finally, we assume that for each $v\in S_p$ there exists a place $\tilde{v}\mid v$ of $F$ such that $\bar{\rho}_{\tilde{v}}\coloneqq\bar{\rho}|_{\G_{F_{\tilde{v}}}}$ is isomorphic to $\bar{r}$ via $F_{\tilde{v}}\iso F^+_v\iso K$. We fix such a place $\tilde{v}$.

Enlarging $L$ if necessary, the existence of such a globalisation has been proved by Emerton-Gee using the Emerton-Gee stack \cite{emertonGee19}; see \cite[Theorem 1.2.2]{emertonGee19} and \cite[Corollary A.7]{emertonGee14}.

For $v\in S_p$, let $R^{\overline{\square}}_{\bar{\rho}_{\tilde{v}}}$ be the maximal reduced quotient of $R^{\square}_{\bar{\rho}_{\tilde{v}}}$ without $p$-torsion. By \cite[Lemma 7.1.4.a]{dejong}, the rigid analytic space $\Xfrak_{\bar{\rho}_{\tilde{v}}}\coloneqq\Spf(R^{\overline{\square}}_{\bar{\rho}_{\tilde{v}}})^\rig$ is reduced and has the same underlying topological space as $\Spf(R^{\square}_{\bar{\rho}_{\tilde{v}}})^\rig$. We also write $\Xfraknotp \coloneqq \prod_{v\in S\setminus S_p}\Xfrak_{\bar{\rho}_{\tilde{v}}}$ and $\Xfrakp \coloneqq \prod_{v\in S_p}\Xfrak_{\bar{\rho}_{\tilde{v}}}$.

Let $T_p\coloneqq\prod_{v\in S_p}T_v$ where $T_v\coloneqq(F_{\tilde{v}}^\times)^n$ is the $n$-dimensional torus over $F_{\tilde{v}}$. Then the continuous characters $T_{v}\to L^\times$ for each $v\in S_p$ (resp.\ $T_p\to L^\times$) are parametrised by a rigid analytic space $\widehat{T}_{v,L}$ (resp.\ $\widehat{T}_{p,L}\coloneqq\prod_{v\in S_p}\widehat{T}_{v,L}$) over $L$. Note that a character $T_v\to L^\times$ can be seen as a character $(K^\times)^n\to L^\times$ through $F_{\tilde{v}}\iso K$; this induces a rigid isomorphism $\widehat{T}_{v,L}\iso\Tcal_L^n$ over $L$. We define $\Xtriv$ in the same way as $\Xtri$ with $\bar{\rho}_{\tilde{v}}$ instead of $\bar{r}$ for each $v\in S_p$, and $\Xtrip\coloneqq\prod_{v\in S_p}\Xtriv$; again the isomorphisms $F_{\tilde{v}}\iso K$ induce an isomorphism $\Xtrip\iso(\Xtri)^{\abs{S_p}}$. Also, since $\Xtrip$ is by definition reduced, there is a Zariski-closed embedding
\[
	\Xtrip \inj \Xfrakp\times_L\widehat{T}_{p,L}
\]
of rigid analytic spaces over $L$.

We fix $g\in\Z_{\geq 1}$, and we let $R_\infty\coloneqq R^{\overline{\square}}_{\bar{\rho}_S}\llbracket t_1,\ldots,t_g\rrbracket$, where $R^{\overline{\square}}_{\bar{\rho}_S}\coloneqq\widehat{\bigotimes}_{v\in S}R^{\overline{\square}}_{\bar{\rho}_{\tilde{v}}}$, and $\Xfrak_\infty\coloneqq\Spf(R_\infty)^\rig$. Then, by \cite[Proposition 7.2.4.g]{dejong}, we have
\[
	\Xfrak_\infty = \Xfraknotp \times_L \Xfrakp \times_L \mathbb{U}^g
\]
where $\mathbb{U}\coloneqq\Spf(\intring_L\llbracket t\rrbracket)^\rig$ is the unit open disk over $L$. By adapting the patching construction of Caraiani-Emerton-Gee-Geraghty-Pa{\v{s}}k{\=u}nas-Shin \cite[\S2]{caraianiEtAl}, Breuil-Hellmann-Schraen constructed for some $g\in\Z_{\geq 1}$ an $R_\infty$-representation $\Pi_\infty$ of $G_p\coloneqq\prod_{v\in S_p}G(F_{\tilde{v}})$ over $L$ called the \og{}patched representation\fg{}, satisfying different properties for which we refer to \cite[Théorème 3.5]{bhs1}. They then consider the subspace $\Pi_\infty^{R_\infty-\mathrm{an}}$ of locally $R_\infty$-analytic vectors (see \cite[Définition 3.2]{bhs1}), and apply Emerton's Jacquet module functor $J_{B_p}$ (see \cite[Definition 3.4.5]{emerton06}) associated to $B_p\coloneqq\prod_{v\in S_p}B_{\tilde{v}}$, where $B_{\tilde{v}}$ is the Borel subgroup of upper triangular matrices in $\GL_n(F_{\tilde{v}})$. This gives a coherent $\func_{\Xfrak_\infty\times_L\widehat{T}_{p,L}}$-module $\mathcal{M}_\infty$ whose module of global sections $\Gamma(\Xfrak_\infty\times_L\widehat{T}_{p,L},\mathcal{M}_\infty)$ is isomorphic to $J_{B_p}(\Pi_\infty^{R_\infty-\mathrm{an}})'$, where $(-)'$ is the continuous dual.

\begin{dfn}[Breuil-Hellmann-Schraen] \label{dfnPatchedEigenvariety}
The \emph{patched eigenvariety} $\Xrho$ is the support of $\mathcal{M}_\infty$ (in the sense of \cite[\S9.5.2]{bgr}), which is an analytic subset of $\Xfrak_\infty\times_L\widehat{T}_{p,L}$ (see Proposition 4 of \emph{ibid.}), equipped with its structure of reduced rigid analytic space over $L$. For more details about this construction, we refer to \cite[\S3.2, Corollaire 3.20]{bhs1} where $\Xrho$ is called the \emph{Hecke-Taylor-Wiles variety}.
\end{dfn}

We next recall the relation between $\Xrho$ and $\Xtrip$. For each $v\in S_p$, define a character $\underline{\delta}_{B_v}=(\abs{\cdot}_{F_{\tilde{v}}}^{n+1-2i})_{1\leq i\leq n}\in\widehat{T}_{v,L}$, where $\abs{\cdot}_{F_{\tilde{v}}}$ is normalised as usual by $\abs{p}_{F_{\tilde{v}}}=p^{-[F_{\tilde{v}}:\Q_p]}$, and an automorphism $\iota_v$ of $\widehat{T}_{v,L}$ as follows:
\[
	\iota_v(\underline{\delta}) \coloneqq \underline{\delta}_{B_v}\cdot(\delta_i\cdot(\chi\circ\theta_{F_{\tilde{v}}}^{-1})^{i-1})_{1\leq i\leq n}
\]
where $\chi$ is the cyclotomic character and $\theta_{F_{\tilde{v}}}\colon W_{F_{\tilde{v}}}^\mathrm{ab}\isoto F_{\tilde{v}}^\times$ is the local reciprocity isomorphism normalised so that geometric Frobeniuses are sent to uniformisers. The inverse of the automorphism
\[
	\appl{\iota}{\Xfrak_\infty\times_L\widehat{T}_{p,L}}{\Xfrak_\infty\times_L\widehat{T}_{p,L}}{\left(y,(\delta_v)_{v\in S_p}\right)}{\left(y,(\iota_v(\delta_v))_{v\in S_p}\right)}
\]
induces a Zariski-closed embedding
\[
	\iota^{-1} \colon \Xrho \inj \Xfraknotp\times_L\left(\Xfrakp\times_L\widehat{T}_{p,L}\right)\times_L\times\mathbb{U}^g \,.
\]
According to \cite[Théorème 3.21]{bhs1}, $\iota^{-1}$ induces an isomorphism of rigid analytic spaces over $L$ between $\Xrho$ and a union of irreducible components of $\Xfraknotp\times_L\Xtrip\times_L\mathbb{U}^g$. Such irreducible components are of the form $\Xfrak^p\times_LZ\times_L\mathbb{U}^g$, where $\Xfrak^p$ is an irreducible component of $\Xfraknotp$ and $Z$ is an irreducible component of $\Xtrip$. Therefore, we have an isomorphism
\begin{equation} \label{eqPatchedToXtri}
	\iota^{-1} \colon \Xrho \isoto \bigcup_{\Xfrak^p\text{ irr.\ comp.\ of }\Xfraknotp}\left(\Xfrak^p\times_L\Xaut\times_L\mathbb{U}^g\right)
\end{equation}
where, for each $\Xfrak^p$, $\Xaut$ is a union of irreducible components of $\Xtrip$.

The following is \cite[Conjecture 3.23]{bhs1}.

\begin{conj} \label{conjBreuilMezard}
For any irreducible component $\Xfrak^p$ of $\Xfraknotp$, the rigid analytic subvariety $\Xaut\subseteq\Xtrip$ is equal to $\Xtriptilde\coloneqq\prod_{v\in S_p}\Xtrivtilde$, where $\Xtrivtilde$ is the union of irreducible components $Z$ of $\Xtriv$ such that $Z\inter\Utriv$ contains a crystalline point. In particular, $\Xaut$ does not depend on $\Xfrak^p$.
\end{conj}

\subsection{Strong linkage and companion points}

We recall a tool from \cite{bhs2} to find points on $\Xrho$ (so-called \emph{companion points}), which is related to the notion of strong linkage of characters which we now explain.

Let $\gfrak$ be a Lie algebra over a field $k$ of characteristic $0$, together with a Cartan subalgebra $\tfrak\subseteq\gfrak$. Let $\Phi\subseteq\tfrak^*\coloneqq\Hom_k(\tfrak,\overline{k})$ be the root system of $(\gfrak,\tfrak)$ (where $\overline{k}$ is an algebraic closure of $k$), let $W\coloneqq W(\gfrak,\tfrak)$ be the Weyl group which acts on $\tfrak^*$ and $\Phi$. Let $E\subset\tfrak^*$ be the $\Q$-span of $\Phi$ and equip $E$ with its nondegenerate bilinear form $(\cdot\mid\cdot)$ invariant by $W$ (see \cite[Chap.\ VI, \S1.1, Prop.\ 3]{bourbakiLie456}). Choose a basis $I$ of $\Phi$ and call its elements \emph{simple roots}; this choice induces a partial order on $\tfrak^*$ by setting $\mu\leq\lambda$ if and only if $\lambda-\mu$ is a $\Z_{\geq0}$-linear combination of simple roots. For each root $\alpha\in\Phi$, let $\alpha^\vee\coloneqq\frac{2\alpha}{(\alpha\mid\alpha)}$ be the coroot associated to $\alpha$. Then $\enstq{\alpha^\vee}{\alpha\in I}$ is a basis of $E$; it has an orthogonal dual under $(\cdot\mid\cdot)$, whose elements we call \emph{fundamental weights}. We call \emph{special weight} the sum of the fundamental weights; it is an element $\varpi\in E$ that satisfies in particular $s_\alpha\varpi=\varpi-\alpha$ for any simple root $\alpha\in I$, where $s_\alpha\in W$ is the reflection associated to $\alpha$, defined by $s_\alpha\colon\lambda\mapsto\lambda-(\lambda\mid\alpha^\vee)\alpha$. Note that $\varpi$ is usually noted $\rho$, \emph{e.g.\ }\cite{BGG}, but we wish to avoid confusion with global representations.

\begin{dfn} \label{dfnDotAction}
The \emph{dot action} is the action of $W$ on $\tfrak^*$ defined by
\[
	w\dotaction\lambda\coloneqq w(\lambda+\varpi)-\varpi
\]
for all $w\in W$ and $\lambda\in\tfrak^*$.

Let $\lambda,\mu\in\tfrak^*$. We say that $\mu$ is \emph{strongly linked} to $\lambda$, and write $\mu\uparrow\lambda$, if there exist roots $\alpha_1,\ldots,\alpha_r\in\Phi$ for some $r\in\Z_{\leq0}$ such that $\mu=(s_{\alpha_r}\ldots s_{\alpha_1})\dotaction\lambda$ and $(s_{\alpha_{i+1}}\ldots s_{\alpha_1})\dotaction\lambda<(s_{\alpha_i}\ldots s_{\alpha_1})\dotaction\lambda$ for all $1\leq i\leq r-1$.
\end{dfn}

We apply this, for each place $v\in S_p$, to the context of the split reductive group $G_{v,L}\coloneqq\left(\Res_{F_{\tilde{v}}/\Q_p}\GL_{n,F_{\tilde{v}}}\right)\times_{\Q_p}L\iso\prod_{\tau\in\Sigma_v}\GL_{n,L}$ over $L$ (where $\Sigma_v\coloneqq\Hom_{\Q_p}(F_{\tilde{v}},L)\iso\Sigma$), with Lie algebra $\gfrak_{v,L}=\prod_{\tau\in\Sigma}\gl_{n,L}$ and Cartan algebra $\tfrak_{v,L}=\prod_{\tau\in\Sigma_v}\tfrak_L$ (with $\tfrak_L\iso L^n$). Then we have $\tfrak_{v,L}^*\iso(\overline{L}^n)^{\Sigma_v}$ as vector spaces over $\overline{L}$, and the Weyl group $W=(\Scal_n)^{\Sigma_v}$ acts on $\tfrak_{v,L}^*$ in the obvious way. We choose the basis $I\subset\Phi$ corresponding to the Borel subgroup of upper triangular matrices; the simple roots are the $\alpha=(\alpha_{\tau',i'})_{\tau'\in\Sigma_v,1\leq i'\leq n}\in\tfrak_{v,L}^*$ such that, for some $\tau\in\Sigma_v$ and $1\leq i\leq n-1$, we have $\alpha_{\tau,i}=-\alpha_{\tau,i+1}=1$ and $\alpha_{\tau',i'}=0$ for all $\tau'\neq\tau$ and $i'\notin\{i,i+1\}$. Then the special weight is $\varpi = \left(\frac{n+1}{2}-i\right)_{\tau\in\Sigma_v,1\leq i\leq n}\in\tfrak_{v,L}^*$, and the dot action expresses as
\[
	w\dotaction\lambda = \left( \lambda_{\tau,w_\tau^{-1}(i)} + i - w_\tau^{-1}(i) \right)_{\tau\in\Sigma_v,1\leq i\leq n}
\]
for $w=(w_\tau)_{\tau\in\Sigma_v}\in W$ and $\lambda=(\lambda_{\tau,i})_{\tau\in\Sigma_v,1\leq i\leq n}\in\tfrak_{v,L}^*$.

If $\underline{\delta}_v=(\delta_{v,i})_{1\leq i\leq n}\colon T_{v,L}\to L'^\times$ is a continuous character, for some place $v\in S_p$ and some finite extension $L'$ of $L$, then one can see $\diff\underline{\delta}_v=(\wt_\tau(\delta_{v,i}))_{\tau\in\Sigma_v,1\leq i\leq n}$ as an element of $\tfrak_{v,L}^*$.

\begin{dfn} \label{dfnStrongLinkage}
Let $\underline{\delta}=(\underline{\delta}_v)_{v\in S_p},\underline{\epsilon}=(\underline{\epsilon}_v)_{v\in S_p}\in\widehat{T}_{p,L}$. We say that $\underline{\epsilon}$ is \emph{strongly linked} to $\underline{\delta}$, if $\diff\underline{\epsilon}_v\uparrow\diff\underline{\delta}_v$ in $\tfrak_{v,L}^*$ for all $v\in S_p$, and if $\underline{\epsilon}\underline{\delta}^{-1}$ is algebraic in the sense that for all $v\in S_p$ there exists $\kbold_v\in\Z^{\Sigma_v}$ such that $\underline{\epsilon}_v\underline{\delta}_v^{-1}=z^{\kbold_v}$. In this case we write $\underline{\epsilon}\uparrow\underline{\delta}$.
\end{dfn}

\begin{thm}[Breuil-Hellmann-Schraen] \label{thmStronglinkagePatched}
Let $x=(y,\underline{\delta})\in\Xfrak_\infty\times_L\widehat{T}_{p,L}$ be a point lying in $\Xrho$, and let $\underline{\epsilon}\in\widehat{T}_{p,L}$ be a character satisfying $\underline{\epsilon}\uparrow\underline{\delta}$. Then the point $x'=(y,\underline{\epsilon})\in\Xfrak_\infty\times_L\widehat{T}_{p,L}$ also lies in $\Xrho$.
\end{thm}

\begin{proof}
This is \cite[Theorem 5.5]{bhs2}; see (3.10) of \emph{loc.\ cit}.
\end{proof}

\section{Tangent spaces on the trianguline variety}

In this section, we prove the main result of the paper: Theorem \ref{thmMain}. It computes the dimension of $\Xtri$ at crystalline points satisfying certain technical conditions, generalising \cite[Corollary 5.17]{bhs2} and \cite[Proposition 4.1.5]{bhs3} (see \cite[Remark 4.1.6(iii)]{bhs3} for the latter).

We keep the notation of \S\ref{secVarieteTrianguline}.

\subsection{Crystalline generic points and pairs of permutations} \label{ssecCrystallineGenericPoints}

We fix a crystalline point $x=(r,\underline{\delta})\in\Xtri$ and assume that $x$ is \emph{generic} in the sense of \cite{bhs3}, \emph{i.e.\ }$r$ is regular and $\varphi_i\varphi_j^{-1}\notin\{1,p^{[K_0:\Q_p]}\}$ for $i\neq j$, where the $1\tens\varphi_i\in K_0\tens_{\Q_p}k(x)$ are the eigenvalues of the linearised Frobenius $\Phi$ on $\Dcris(r)$. We write $\hbold=(h_{\tau,i})_{\tau\in\Sigma,1\leq i\leq n}$ for its Hodge-Tate weights, ordered such that $h_{\tau,1}<\ldots<h_{\tau,n}$ for all $\tau\in\Sigma$.

By Lemma \ref{lemXtriPoints}, we can attach two elements $w,w_\sat\in\Scal_n^\Sigma$ such that (i) the character $\underline{\delta}$ is of the form $\underline{\delta}=z^{w(\hbold)}\nr(\underline{\varphi})$, where $\underline{\varphi}=(\varphi_i)_{1\leq i\leq n}$ and $w(\hbold)\coloneqq(h_{\tau,w_\tau^{-1}(i)})_{\tau,i}\in(\Z^\Sigma)^n$, and (ii) the point $x_\sat\coloneqq(r,\underline{\delta}_\sat)\in\Xfrakr\times_L\Tcal_L^n$ defined by $\underline{\delta}_\sat=z^{w_\sat(\hbold)}\nr(\underline{\varphi})$ lies in $\Utri$.

Remember that $(\Scal_n)^\Sigma$ is the Weyl group of the root system $\coprod_{\tau\in\Sigma}\Phi$ attached to the split reductive group $\left(\prod_{\tau\in\Sigma}G,\prod_{\tau\in\Sigma}T\right)$ over $\Q_p$ where $G\coloneqq\GL_{n,\Q_p}$ and $T\subseteq G$ is the torus of diagonal matrices. Writing $X(T)\coloneqq\Hom(T,\Gm)$ for the $\Z$-module of algebraic characters of $T$, one has $\Phi\subset X(T)$ and the canonical isomorphism $X(T)\iso\Z^n$ induces $\Phi\iso\enstq{e_i-e_j}{1\leq i,j\leq n\;,\;i\neq j}$ where $(e_i)_{1\leq i\leq n}$ is the standard basis of $\Z^n$. For $w'\in(\Scal_n)^\Sigma$, we denote by $d_{w'}$ the rank of the $\Z$-submodule of $X\left(\prod_{\tau\in\Sigma}T\right)$ generated by $\enstq{w'(\alpha)-\alpha}{\alpha\in\coprod_{\tau\in\Sigma}\Phi}$. We choose the standard basis $I\coloneqq\enstq{e_i-e_{i+1}}{1\leq i\leq n-1}$ of $\Phi$ corresponding to the Borel subgroup $B\subseteq G$ of upper triangular matrices. It induces on $\Scal_n$ a length function $\lg$ and a partial Bruhat order $\preceq$; these extend naturally to a length function and a Bruhat order on $(\Scal_n)^\Sigma$, still denoted $\lg$ and $\preceq$, which can also be defined directly by choosing the corresponding standard basis $\coprod_{\tau\in\Sigma}I$ of the root system $\coprod_{\tau\in\Sigma}\Phi$. By \cite[Theorem 4.2.3]{bhs3}, we have $w_\sat\preceq w$.

\begin{dfn} \label{dfnGoodPair}
Let $(w_1,w_2)\in(\Scal_n)^2$ such that $w_1\preceq w_2$. We say that $(w_1,w_2)$ is a \emph{good pair} if there exists a sequence $(\alpha_1,\ldots,\alpha_r)$ of roots of the form $e_i-e_j$, where $i,j\in\{1,\ldots,n\}$ are in the same orbit under the action of $w_1w_2^{-1}$, such that $w_2=s_{\alpha_r}s_{\alpha_{r-1}}\ldots s_{\alpha_1}w_1$ and $s_{\alpha_i}\ldots s_{\alpha_1}w_1\succ s_{\alpha_{i-1}}\ldots s_{\alpha_1}w_1$ for all $1\leq i\leq r$.

For $(w_1,w_2)\in\left((\Scal_n)^\Sigma\right)^2$ such that $w_1\preceq w_2$, we say that $(w_1,w_2)$ is a \emph{good pair} if and only if $(w_{1,\tau},w_{2,\tau})\in(\Scal_n)^2$ is a good pair for all $\tau\in\Sigma$.
\end{dfn}

\begin{ex} \label{exGoodPairs}
For $n\leq4$, all pairs are good except for $(w_1,w_2)\in(\Scal_4)^2$ defined by $(w_1(1),w_1(2),w_1(3),w_1(4))=(1,3,2,4)$ and $(w_2(1),w_2(2),w_2(3),w_2(4))=(4,2,3,1)$. In fact, we show in a subsequent paper \cite[Theorem 2.4.10]{chap245} a result relating good pairs to pattern avoidance, which implies in particular that ``most'' pairs are good.
\end{ex}

From now on we assume the following:

\begin{hyp} \label{hypComesFromPatched}
With the notation of \S\ref{secVarieteTrianguline}, there exists an irreducible component $\Xfrak^p$ of $\Xfraknotp$ such that $(x_{\sat,v})_{v\in S_p}\in\Xaut$, where $x_{\sat,v}\in\Xtriv$ corresponds to $x_\sat\in\Xtri$ under $\bar{\rho}_{\tilde{v}}\iso\bar{r}$ for each $v\in S_p$.
\end{hyp}

Note that Hypothesis \ref{hypComesFromPatched} is implied by Conjecture \ref{conjBreuilMezard} since $x_{\sat,v}\in\Xtrivtilde$ by definition.

\begin{thm} \label{thmMain}
Assume Hypothesis \ref{hypComesFromPatched}. Let $x\in\Xtri$ be crystalline generic, such that the associated pair $(w_\sat,w)$ in $(\Scal_n)^\Sigma$ is good in the sense of Definition \ref{dfnGoodPair}. Then
\[
	\dim_{k(x)}T_{\Xtri,x} = \dim\Xtri - d_{ww_\sat^{-1}} + \sum_{\tau\in\Sigma}\dim_{\Q_p}T_{\overline{(Bw_\tau B/B)},w_{\sat,\tau}B} - \length(w_\sat)
\]
where $d_{ww_\sat^{-1}}$ is as defined in the paragraph preceding Definition \ref{dfnGoodPair}.
\end{thm}

The rest of the section is devoted to the proof of Theorem \ref{thmMain}. We use a similar strategy to \cite{bhs2}, as follows. We construct two subspaces of $T_{\Xtri,x}$, one roughly corresponding to ``crystalline tangent directions'' and the other to ``trianguline tangent directions'', and compute the dimension of the sum of these subspaces: this bounds $\dim_{k(x)}T_{\Xtri,x}$ from below. The upper bound is proved in \cite[Theorem 4.1.5(ii)]{bhs3}.

\subsection{Crystalline tangent directions}

Following the exposition of \cite[\S2.2]{bhs2} and \cite[\S4.2]{bhs3}, we recall the construction of a rigid analytic space over $L$ parametrising crystalline deformations of given Hodge-Tate weights together with an ordering of the eigenvalues of the Frobenius.

Let $\Rcr$ be the quotient of $\Rr$ representing framed deformations of $\bar{r}$ which are crystalline of Hodge-Tate weights $\hbold$ (its existence is guaranteed by \cite[Corollary 2.6.2]{kisin08}; it is reduced and flat over $\Z_p$). Namely, for any complete local noetherian algebra $A$ over $\intring_L$ with residue field $k_L$, a morphism $\Rr\to A$ factors through $\Rcr$ if and only if the corresponding deformation $r\colon\G_K\to\GL_n(A)$ is crystalline with Hodge-Tate weights $\hbold$. Write $\Xfrakcr \coloneqq \Spf(\Rcr)^\rig$ for the associated rigid analytic space over $L$. Arguing as in \cite[Proposition 2.3.5]{kisin09} similarly to the case of $\Xfrakr$ (see \S\ref{secVarieteTrianguline}), we see that for any local artinian algebra $A$ over $L$, the set of $A$-points $x\in\Xfrakcr(A)$ is the set of points $x\in\Xfrakr(A)$ such that the associated representation $r_x\colon\G_K\to\GL_n(A)$ is crystalline with Hodge-Tate weights $\hbold$.

According to \cite[Theorem 2.5.5]{kisin08} (see also \cite[Corollaire 6.3.3]{bergerColmez}), there is a coherent locally free module $\Dcal$ over $K_0\tens_{\Q_p}\func_{\Xfrakcr}$ equipped with a $\varphi$-semilinear automorphism $\varphi_\cris$ such that, for all $x\in\Xfrakcr$ corresponding to a deformation $r_x\colon\G_K\to\GL_n(k(x))$, there is an isomorphism
\[
	\Dcal\tens_{\func_{\Xfrakcr}}k(x) \iso \Dcris(r_x)
\]
such that $\varphi_\cris\tens_{\func_{\Xfrakcr}}k(x)$ induces the $\varphi$-action on $\Dcris(r_x)$. Furthermore, for each $\tau\in\Hom_{\Q_p}(K_0,L)$, the linearised Frobenius $\Phi_\cris\coloneqq(\varphi_\cris)^{[K_0:\Q_p]}$ induces an invertible linear action on $\Dcris(r_x)\tens_{K_0,\tau}L$ which is independent, up to isomorphism, of the choice of $\tau$. Therefore, taking the coefficients of the characteristic polynomial of this action gives a morphism of rigid analytic spaces over $L$
\begin{equation} \label{eqXfrakcrCharPoly}
	\Xfrakcr \to \left((\Ga^\rig)^{n-1}\times\Gm^\rig\right)\times_{\Q_p}L
\end{equation}
where the image $(a_{n-1},\ldots,a_1,a)\in k(x)^{n-1}\times k(x)^\times$ of a point $x\in\Xfrakcr$ is such that the Frobenius on $\Dcris(r_x)$ has eigenvalues $1\tens\varphi_i\in K_0\tens_{\Q_p}k(x)$, where the $\varphi_i$ for $1\leq i\leq n$ are the roots of the polynomial $t^n+a_{n-1}t^{n-1}+\ldots+a_1t+a$. See also \cite[\S2.2]{bhs2}.

As in the paragraph following Definition \ref{dfnFlagModules}, let $G\coloneqq\GL_{n,\Q_p}$, $B\subset G$ be the Borel of upper triangular matrices, $T\subset B$ be the maximal torus inside $B$; the Weyl group of $(G,T)$ is the symmetric group $\Scal_n$. The isomorphism $L[e_1,\ldots,e_n,f]/(e_nf-1)\isoto L[t_1,\ldots,t_n,u]^{\Scal_n}/(u\prod t_i-1)$ of the fundamental theorem of symmetric polynomials gives an isomorphism of schemes
\[
	T_L/\Scal^n \isoto \left(\Ga^{n-1}\times\Gm\right)\times_{\Q_p}L
\]
where the left-hand side parametrises a polynomial by its roots and the right-hand side parametrises it by its coefficients. Therefore \eqref{eqXfrakcrCharPoly} is a morphism $\Xfrakcr\to T_L^\rig/\Scal^n$ of rigid analytic spaces over $L$, which we use to define
\[
	\Xfrakcrtilde \coloneqq \Xfrakcr\times_{T_L^\rig/\Scal^n}T_L^\rig
\]
which is a reduced rigid analytic space over $L$ by \cite[Lemma 2.2]{bhs2}.

From the proof of \cite[Theorem 4.2.3]{bhs3} (see also \cite[Lemma 2.4]{bhs2}), the set
\[
	\Wcrtilde \coloneqq \enstq{(r',\underline{\varphi'})\in\Xfrakcrtilde}{\varphi'_i{\varphi'_j}^{-1}\notin\{1,p^{[K_0:\Q_p]}\}\;\forall\;i\neq j}
\]
is a Zariski-open dense subset of $\Xfrakcrtilde$, and there is a smooth morphism
\begin{equation} \label{eqDfnh}
	h \colon \Wcrtilde \to \left(\Res_{K/\Q_p}(G/B)_K^\rig\right)\times_{\Q_p}L \iso \prod_{\tau\in\Sigma} (G/B)_L^\rig
\end{equation}
of rigid analytic spaces over $L$ sending a point $(r',\underline{\varphi'})\in\Wcrtilde$ to the position of the Hodge flag $\Fil_\bullet\DdR(r')$ relative to the refinement on $K\tens_{K_0}\Dcris(r')$ induced by the ordering $\underline{\varphi'}$ (see the discussion following Definition \ref{dfnRefinement}). Note that we use the notation $r'$ and $\underline{\varphi'}$ to avoid confusion with $r$ and $\underline{\varphi}$ which are defined in \S\ref{ssecCrystallineGenericPoints}.

Fix a permutation $w=(w_\tau)\in\Scal_n^\Sigma$ and define the analytic subset
\[
	\Wcrtildew{w} \coloneqq h^{-1}\left(\prod_{\tau\in\Sigma}(\overline{Bw_\tau B/B})_L^\rig\right) \subseteq \Wcrtilde
\]
of $\Wcrtilde$ (note that it is written $\overline{\Wcrtildew{w}}$ in \cite[\S4.2]{bhs3}; we drop the bar for clarity).

\begin{prop}[Breuil-Hellmann-Schraen] \label{propIotahw}
The morphism
\[
	\appl{\iota_{\hbold,w}}{\Wcrtildew{w}}{\Xfrakr\times_L\Tcal_L^n}{\left(r',\underline{\varphi'}\right)}{\left(r',z^{w(\hbold)}\nr(\underline{\varphi'})\right)}
\]
 of rigid analytic spaces over $L$ (where we recall that $w(\hbold)\coloneqq(h_{\tau,w_\tau^{-1}(i)})_{\tau,i}\in(\Z^\Sigma)^n$) is a closed immersion and factors through the inclusion $\Xtri\subset\Xfrakr\times_L\Tcal_L^n$.
\end{prop}

\begin{proof}
See the proof of \cite[Theorem 4.2.3]{bhs3}.
\end{proof}

Note that $\iota_{\hbold,w}(r,\underline{\varphi})=x$, therefore $\iota_{\hbold,w}$ induces a morphism $T_{\Wcrtildew{w},(r,\underline{\varphi})}\inj T_{\Xtri,x}$.

\subsection{Trianguline tangent directions}

We generalise a construction of \cite[\S5.2]{bhs2}

Write $\Wcal^n_{w,w_\sat,\hbold,L}\subset\Wcal^n_L$ for the analytic subset of characters $\underline{\eta}=(\eta_1,\ldots,\eta_n)$ defined by the equations:
\[
	\wt_\tau\big(\eta_{w_{\sat,\tau}(i)}\eta_{w_\tau(i)}^{-1}\big) = h_{\tau,i}-h_{\tau,w_{\sat,\tau}^{-1}w_\tau(i)} \,, \quad 1\leq i\leq n \,, \quad \tau\in\Sigma
\]
with its structure of reduced rigid analytic space over $L$. Such a character $\underline{\eta}$ is of the form $z^{w_\sat(\hbold)}\underline{\eta}'$ where $\underline{\eta}'=(\eta'_1,\dots,\eta'_n)\in\Wcal_L^n$ satisfies $\wt_\tau\big(\eta'_{w_\tau w_{\sat,\tau}^{-1}(i)}\big)=\wt_\tau(\eta'_i)$ for all $\tau,i$. In particular, this is the case of $\underline{\delta}_\sat$.

We define an automorphism $\jmathauto\colon\Tcal^n_L\isoto\Tcal^n_L$ by
\[
	\jmathauto(\underline{\eta}) \coloneqq z^{w(\hbold)-w_\sat(\hbold)}\underline{\eta}
\]
and still write $\jmathauto$ for the automorphism $\id_{\Xfrakr}\tens\jmathauto\colon\Xfrakr\times_L\Tcal^n_L\isoto\Xfrakr\times_L\Tcal^n_L$.

By \cite[Corollary 3.7.10]{bhs3}, the rigid variety $\Xtri$ is locally irreducible at $x_\sat$ so, with the notation of Hypothesis \ref{hypComesFromPatched}, we can identify $\Xaut$ to $\Xtrip$ locally at $(x_{\sat,v})_{v\in S_p}$. Therefore, there exists a Zariski-open neighbourhood
\[
	U_{x_\sat}\subseteq\Utri
\]
of $x_\sat$ such that $\prod_{v\in S_p}U_{x_{\sat,v}}\subseteq\Xaut$, where each $U_{x_{\sat,v}}\subset\Xtriv$ corresponds to $U_{x_\sat}\subset\Xtri$ under $\bar{\rho}_{\tilde{v}}\iso\bar{r}$.

Since $U_{x_\sat}$ is a Zariski-open subset of $\Utri$, by Theorem \ref{thmXtriGeometry} and by base change, $U_{x_\sat}\times_{\Wcal_L^n}\Wcal^n_{w,w_\sat,\hbold,L}$ is smooth over $\Wcal^n_{w,w_\sat,\hbold,L}$ hence reduced, and it is a Zariski-open subset of $\Xtri\times_{\Wcal^n_L}\Wcal^n_{w,w_\sat,\hbold,L}$. Consider its Zariski-closure $\overline{U_{x_\sat}\times_{\Wcal_L^n}\Wcal^n_{w,w_\sat,\hbold,L}}$, with its structure of reduced rigid analytic space over $L$; it fits in a chain of Zariski-closed embeddings
\begin{multline*}
	\overline{U_{x_\sat}\times_{\Wcal^n_L}\Wcal^n_{w,w_\sat,\hbold,L}} \inj (\Xtri\times_{\Wcal^n_L}\Wcal^n_{w,w_\sat,\hbold,L})^\red \\
	\inj \Xtri\times_{\Wcal^n_L}\Wcal^n_{w,w_\sat,\hbold,L} \inj \Xtri \subseteq \Xfrakr\times_L\Tcal^n_L
\end{multline*}
and the automorphism $\jmathauto\colon\Xfrakr\times_L\Tcal^n_L\isoto\Xfrakr\times_L\Tcal^n_L$ induces a Zariski-closed embedding
\begin{equation} \label{eqJmath}
	\jmathauto \colon \overline{U_{x_\sat}\times_{\Wcal^n_L}\Wcal^n_{w,w_\sat,\hbold,L}} \inj \Xfrakr\times_L\Tcal^n_L \,.
\end{equation}

\begin{prop} \label{propDimMinusD}
We have
\[
	\dim_{k(x)}T_{U_{x_\sat}\times_{\Wcal^n_L}\Wcal^n_{w,w_\sat,\hbold,L},x_\sat}=\dim{\Xtri}-d_{ww_\sat^{-1}}
\]
(see the paragraph preceding Definition \ref{dfnGoodPair} for the definition of $d_{ww_\sat^{-1}}$).
\end{prop}

\begin{proof}
This is the same as \cite[Proposition 5.15]{bhs2}.
\end{proof}

The following result is adapted from \cite[Proposition 5.9]{bhs2}.

\begin{prop} \label{propClosedEmbedding}
Assume that the pair $(w_\sat,w)$ in $(\Scal_n)^\Sigma$ is good in the sense of Definition \ref{dfnGoodPair}. Then \eqref{eqJmath} induces a Zariski-closed embedding of reduced rigid analytic spaces over $L$
\begin{equation} \label{closedEmbedding}
	\jmathauto \colon \overline{U_{x_\sat}\times_{\Wcal^n_L}\Wcal^n_{w,w_\sat,\hbold,L}} \inj \Xtri \,.
\end{equation}
\end{prop}

\begin{proof}
It is enough to prove that any point $x'=(r',\underline{\delta}')\in U_{x_\sat}$ such that $\underline{\delta}'|_{(\intring_K^\times)n}\in\Wcal^n_{w,w_\sat,\hbold,L}$ satisfies $\jmathauto(x')\in\Xtri$. By definition, the open subset $\prod_{v\in S_p}U_{x_{\sat,v}}$ of $\Xtrip$ is in the image of the morphism
\begin{equation} \label{eqPatchedToTri2}
	\Xrho \inj \Xfraknotp\times_L\Xtrip\times_L\times\mathbb{U}^g \surj \Xtrip
\end{equation}
where the first arrow is \eqref{eqPatchedToXtri}.
Let $y'=(s',\underline{\epsilon}')\in\Xfrak_\infty\times\widehat{T}_{p,L}$ be a point of $\Xrho$ whose image by \eqref{eqPatchedToTri2} is $(x'_v)_{v\in S_p}$, where $x'_v\in U_{x_{\sat,v}}$ corresponds to $x'\in U_{x_\sat}$ under $U_{x_{\sat,v}}\iso U_{x_\sat}$. We know in particular that $\underline{\epsilon}'=(\iota_v(\underline{\delta}'_v))_{v\in S_p}$, where $\underline{\delta}'_v$ is the composition of $(F_{\tilde{v}}^\times)^n\isoto(K^\times)^n$ with $\underline{\delta}'$. Consider the point
\[
	\jmathauto(y') \coloneqq (s',\jmathauto(\underline{\epsilon}')) \in\Xfrak_\infty\times\widehat{T}_{p,L}
\]
where $\jmathauto(\underline{\epsilon}')\coloneqq(\jmathauto(\iota_v(\underline{\delta}_v')))_{v\in S_p}$. It is enough to prove that $\jmathauto(y')\in\Xrho$, since $(\jmathauto(x'_v))_{v\in S_p}$ would then be the image of $\jmathauto(y')$ under \eqref{eqPatchedToTri2} (as $\jmathauto$ and $\iota_v$ commute), hence a point of $\Xtrip=\prod_{v\in S_p}\Xtriv$.
 
 By Theorem \ref{thmStronglinkagePatched}, all we need to prove is that $\underline{\epsilon}'\uparrow\jmathauto(\underline{\epsilon}')$ in the sense of Definition \ref{dfnStrongLinkage}. Since $\underline{\epsilon}'\jmathauto(\underline{\epsilon}')^{-1}=z^{w_\sat(\hbold)-w(\hbold)}$ is algebraic by definition of $\jmathauto$, it is equivalent to check that
\begin{equation} \label{eqUpArrowEpsilon}
	(\wt_\tau(\epsilon_{v,i}))_{\tau\in\Sigma_v,1\leq i\leq n} \uparrow \left(h_{\tau,w_v^{-1}(i)}-h_{\tau,w_{\sat,v}^{-1}}(i)+\wt_\tau(\epsilon_{v,i})\right)_{\tau\in\Sigma_v,1\leq i\leq n}
\end{equation}
as elements of $\mathfrak{t}_{v,L}^*$, for all $v\in S_p$. Viewing $\underline{\delta}'$ as its restriction to $(\intring_K^\times)^n$, we can write:
\[
	\underline{\delta}'=(z^{\hbold_{w_\sat^{-1}(1)}}\chi_1,\ldots,z^{\hbold_{w_\sat^{-1}(n)}}\chi_n)
\]
where $\underline{\chi}\in\Wcal_L^n$ satisfies $\wt_\tau(\chi_{w_\tau w_{\sat,\tau}^{-1}(i)})=\wt_\tau(\chi_i)$ for all $1\leq i\leq n$ and $\tau\in\Sigma$. Unpacking the definitions and using the fact that $\chi\circ\theta_K^{-1}=N_{K/\Q_p}\abs{N_{K/\Q_p}}_{\Q_p}$ has the same restriction to $\intring_K^\times$ as $N_{K/\Q_p}=z^{\mathbf{1}}$, we see that \eqref{eqUpArrowEpsilon} can be reformulated as 
\[
	\left(h_{\tau,w_{\sat,\tau}^{-1}(i)}+i-1+\wt_\tau(\chi_i)\right)_{\tau\in\Sigma,1\leq i\leq n} \uparrow \left(h_{\tau,w_\tau^{-1}(i)}+i-1+\wt_\tau(\chi_i)\right)_{\tau\in\Sigma,1\leq i\leq n} \,.
\]
Since $\wt_\tau(\chi_{w_\tau w_{\sat,\tau}^{-1}(i)})=\wt_\tau(\chi_i)$ for all $i$, this equation can also be written
\begin{equation} \label{eqUpArrowDot}
	w_\sat\cdot\lambda \uparrow w\cdot\lambda \quad\text{with}\quad \lambda\coloneqq(h_{\tau,i}+i-1+\wt_\tau(\chi_{w_\tau(i)}))_{\tau\in\Sigma,1\leq i\leq n}
\end{equation}
where $\cdot$ is the dot action of Definition \ref{dfnDotAction}.

Since $(w_\sat,w)$ is a good pair in $(\Scal_n)^\Sigma$, there exists a chain $w_0\coloneqq w_\sat\prec w_1\prec\ldots\prec w_r\coloneqq w$ in $(\Scal_n)^\Sigma$ for some $r\in\N$ such that for each $1\leq k\leq r$ and $\tau\in\Sigma$, the permutation $w_{k,\tau} w_{k-1,\tau}^{-1}\in\Scal_n$ is either the identity or a reflection which is supported in some orbit of $w_\tau w_{\sat,\tau}^{-1}$. By the condition on the weights of $\underline{\chi}$, this means in particular that $\wt_\tau(\chi_{w_{k,\tau}w_{k-1,\tau}^{-1}(i)})=\wt_\tau(\chi_i)$ for each $\tau,k,i$. Thus, by a descending recursion on $k$, we see that
\[
	\wt_\tau\left(\chi_{w_\tau w_{k,\tau}^{-1}(i)}\right)=\wt_\tau\left(\chi_i\right)=\wt_\tau\left(\chi_{w_\tau w_{k-1,\tau}^{-1}(i)}\right)
\]
for each $\tau,k,i$. We then get, for all $1\leq k\leq n$,
\[
	w_k\cdot\lambda - w_{k-1}\cdot\lambda = \left(h_{\tau,w_{k,\tau}^{-1}(i)}-h_{\tau,w_{k-1,\tau}^{-1}(i)}\right)_{\tau\in\Sigma,1\leq i\leq n}
\]
thus $w_k\cdot\lambda>w_{k-1}\cdot\lambda$ (remember that $>$ is the partial order on $\tfrak_L^*$ induced by the lattice of simple roots) since for each $\tau\in\Sigma$, on the one hand $w_{k-1,\tau}^{-1}\prec w_{k,\tau}^{-1}$ and $w_{k-1,\tau}w_{k,\tau}^{-1}$ is at most a reflection, and on the other hand $(h_{\tau,i})_{1\leq i\leq n}$ is strictly increasing. This means that $w_{k-1}\cdot\lambda\uparrow w_k\cdot\lambda$ for each $1\leq k\leq n$, hence by chaining these together $w_\sat\cdot\lambda\uparrow w\cdot\lambda$. This is exactly \eqref{eqUpArrowDot}.
\end{proof}

Note that $\jmathauto(x_\sat)=x$, therefore $\jmathauto$ induces an injective morphism $T_{\overline{U_{x_\sat}\times_{\Wcal^n_L}\Wcal^n_{w,w_\sat,\hbold,L}},x_\sat}\inj T_{\Xtri,x}$.

\subsection{Proof of Theorem \ref{thmMain}}

We now prove Theorem \ref{thmMain}. It is equivalent to proving, assuming Hypothesis \ref{hypComesFromPatched} and that $(w_\sat,w)$ is a good pair in $(\Scal_n)^\Sigma$ in the sense of Definition \ref{dfnGoodPair}, that
\[
	\dim_{k(x)}T_{\Xtri,x} = \dim\Xtri - d_{ww_\sat^{-1}} + \sum_{\tau\in\Sigma}\dim_{k(x)}T_{\overline{(Bw_\tau B/B)}_{k(x)}^\rig,x_\tau} - \length(w_\sat)
\]
where $(x_\tau)_{\tau\in\Sigma}\coloneqq h(r,\underline{\varphi})\in\prod_{\tau\in\Sigma}(Bw_{\sat,\tau}B/B)_L^\rig$ (see \eqref{eqDfnh} for the definition of $h$). Indeed, the dimension of $T_{\overline{Bw_\tau B/B},y_\tau B}$ is constant as $y_\tau$ varies through $Bw_{\sat,\tau}B/B$.

The inequality $(\text{left-hand side})\leq(\text{right-hand side})$ is the statement of \cite[Theorem 4.1.5(ii)]{bhs3} (where our $w_\sat$ is noted $w_x$). We prove the inequality in the other direction. Consider the following commutative diagram:
\begin{equation} \label{eqCommutativeDiagramWcrtilde}
\begin{tikzcd}
	\widetilde{W}_{x_\sat} \arrow[r,hook,"\iota_{\hbold,w_\sat}"] \arrow[d,hook,"\subseteq"] & \overline{U_{x_\sat}\times_{\Wcal^n_L}\Wcal^n_{w,w_\sat,\hbold,L}} \arrow[d,hook,"\jmathauto"] \\
	\Wcrtildew{w} \arrow[r,hook,"\iota_{\hbold,w}"] & \Xtri
\end{tikzcd}
\end{equation}
where $\iota_{\hbold,w}$ and $\iota_{\hbold,w_\sat}$ are given in Proposition \ref{propIotahw}, $\jmathauto$ is given in Proposition \ref{propClosedEmbedding} and $\widetilde{W}_{x_\sat}\coloneqq\iota_{\hbold,w_\sat}^{-1}\left(\overline{U_{x_\sat}\times_{\Wcal^n_L}\Wcal^n_{w,w_\sat,\hbold,L}}\right)$ is an analytic subset of $\Wcrtildew{w_\sat}$, hence an analytic subset of $\Wcrtildew{w}$. From the definitions, we see that $\iota_{\hbold,w_\sat}\colon\Wcrtildew{w_\sat}\to\Xtri$ factors through $\Xtri\times_{\Wcal^n_L}\Wcal^n_{w,w_\sat,\hbold,L}\to\Xtri$. Since $\iota_{\hbold,w_\sat}(r,\underline{\varphi})=x_\sat\in U_{x_\sat}\times_{\Wcal^n_L}\Wcal^n_{w,w_\sat,\hbold,L}$ and $U_{x_\sat}\times_{\Wcal^n_L}\Wcal^n_{w,w_\sat,\hbold,L}$ is a Zariski-open subset of $\Xtri\times_{\Wcal^n_L}\Wcal^n_{w,w_\sat,\hbold,L}$, it follows that $T_{\widetilde{W}_{x_\sat},(r,\underline{\varphi})}=T_{\Wcrtildew{w_\sat},(r,\underline{\varphi})}$. Therefore, \eqref{eqCommutativeDiagramWcrtilde} induces a commutative diagram of tangent spaces
\begin{equation} \label{eqTangentSpacesCartesianSquare}
\begin{tikzcd}
	T_{\Wcrtildew{w_\sat},(r,\underline{\varphi})} \arrow[r,hook] \arrow[d,hook] & T_{\overline{U_{x_\sat}\times_{\Wcal^n_L}\Wcal^n_{w,w_\sat,\hbold,L}},x_\sat} \arrow[d,hook] \\
	T_{\Wcrtildew{w},(r,\underline{\varphi})} \arrow[r,hook] & T_{\Xtri,x} \,.
\end{tikzcd}
\end{equation}
Let $T_1,T_2\subseteq T_{\Xtri,x}$ be the respective images of the bottom and right arrows of \eqref{eqTangentSpacesCartesianSquare}. Then this diagram induces a chain of morphisms
\begin{equation} \label{eqCartesianSquare}
	T_{\Wcrtildew{w_\sat},(r,\underline{\varphi})} \to T_{\Wcrtildew{w},(r,\underline{\varphi})} \times_{T_{\Xtri,x}} T_{\overline{U_{x_\sat}\times_{\Wcal^n_L}\Wcal^n_{w,w_\sat,\hbold,L}},x_\sat} \isoto T_1\inter T_2
\end{equation}
of vector spaces.

The first morphism in \eqref{eqCartesianSquare} is injective since $T_{\Wcrtildew{w_\sat},(r,\underline{\varphi})}\to T_{\Wcrtildew{w},(r,\underline{\varphi})}$ is injective; we then wish to prove its surjectivity. Let $A=k(x)[\varepsilon]/(\varepsilon^2)$ be the algebra of infinitesimal numbers over $k(x)$. An element $v_\sat\in T_{U_{x_\sat},x_\sat}$ is a pair $v_\sat=(r_A,\underline{\delta}_{\sat,A})$ where $r_A\colon\G_K\to\GL_n(A)$ is a representation which reduces to $r$ modulo $(\varepsilon)$ and $\underline{\delta}_{\sat,A}$ is a character $(K^\times)^n\to A^\times$ which reduces to $\underline{\delta}_\sat$ modulo $(\varepsilon)$. We know from \cite[Corollary 6.3.10]{kpx} that the triangulation on $r$ globalises in a neighbourhood of $x_\sat$ in $\Utri$, hence $r_A$ is triangular with parameter $\underline{\delta}_{\sat,A}$; see the argument in the proof of \cite[Proposition 5.16]{bhs2}. Therefore, an element $v$ in the right-hand side of \eqref{eqCartesianSquare} is a triple $v=(r_A,\underline{\varphi}_A,\underline{\delta}_A)$ where $r_A\colon\G_K\to\GL_n(A)$ is a crystalline representation of Hodge-Tate weights $\hbold$ which reduces to $r$ modulo $(\varepsilon)$, $\underline{\varphi}_A\in (A^\times)^n$ is the ordering of the eigenvalues of $r_A$ which reduces to $\underline{\varphi}$ modulo $(\varepsilon)$, and $\underline{\delta}_A$ is the parameter of a triangulation of $r_A$, such that
\[
	z^{w(\hbold)}\nr(\underline{\varphi}_A) = \jmathauto(\underline{\delta}_A) = z^{w(\hbold)-w_\sat(\hbold)}\underline{\delta}_A \,.
\]
In particular, $\underline{\delta}=z^{w_\sat(\hbold)}\nr(\underline{\varphi})$; hence Proposition \ref{propLikeBellaicheChenevier}\ref{itemFlagTriangulation} implies that $(r_A,\underline{\varphi}_A)\in\Wcrtildew{w_\sat}$, which is to say that $v$ is in the image of \eqref{eqCartesianSquare}. This proves that \eqref{eqCartesianSquare} is a chain of isomorphisms of vector spaces.

Therefore,
\[
	\dim_{k(x)}(T_1\inter T_2) = \dim_{k(x)}T_{\Wcrtildew{w_\sat},(r,\underline{\varphi})}
\]
thus
\begin{multline} \label{eqDim1}
	\dim_{k(x)}T_{\Xtri,x} \geq \dim_{k(x)}(T_1+T_2) = \dim_{k(x)}T_{\overline{U_{x_\sat}\times_{\Wcal^n_L}\Wcal^n_{w,w_\sat,\hbold,L}},x_\sat}
	\\ + \dim_{k(x)}T_{\Wcrtildew{w},(r,\underline{\varphi})} - \dim_{k(x)}T_{\Wcrtildew{w_\sat},(r,\underline{\varphi})} \,.
\end{multline}
By Proposition \ref{propDimMinusD}, we have
\begin{equation} \label{eqDim2}
	\dim_{k(x)}T_{\overline{U_{x_\sat}\times_{\Wcal^n_L}\Wcal^n_{w,w_\sat,\hbold,L}},x_\sat}=\dim\Xtri - d_{ww_\sat^{-1}} \,.
\end{equation}
Also, the cartesian diagram
\[
\begin{tikzcd}
	\Wcrtildew{w_\sat} \arrow[r,"h"] \arrow[d,"\subseteq"] & \prod_{\tau\in\Sigma}\overline{(Bw_{\sat,\tau}B/B)}_L^\rig \arrow[d,"\subseteq"] \\
	\Wcrtildew{w} \arrow[r,"h"] & \prod_{\tau\in\Sigma}\overline{(Bw_\tau B/B)}_L^\rig
\end{tikzcd}
\]
and the smoothness of $h$ yield
\begin{multline} \label{eqDim3}
	\dim_{k(x)}T_{\Wcrtildew{w},(r,\underline{\varphi})} - \dim_{k(x)}T_{\Wcrtildew{w_\sat},(r,\underline{\varphi})} \\
	= \sum_{\tau\in\Sigma}\dim_{k(x)}T_{\overline{(Bw_\tau B/B)}_{k(x)}^\rig,x_\tau} - \sum_{\tau\in\Sigma}\dim_{k(x)}T_{\overline{(Bw_{\sat,\tau}B/B)}_{k(x)}^\rig,x_\tau} \,.
\end{multline}
Finally, as $(x_\tau)$ lies in the smooth cell $\prod_{\tau\in\Sigma}(Bw_{\sat,\tau}B/B)_L^\rig$, we have
\begin{equation} \label{eqDim4}
	\sum_{\tau\in\Sigma}\dim_{k(x)}T_{\overline{(Bw_{\sat,\tau}B/B)}_{k(x)}^\rig,x_\tau} = \dim\prod_{\tau\in\Sigma}(Bw_{\sat,\tau}B/B)_L^\rig = \length(w_\sat) \,.
\end{equation}
Putting together \eqref{eqDim1}, \eqref{eqDim2}, \eqref{eqDim3}, \eqref{eqDim4} gives the desired inequality, which finishes the proof.

\bibliographystyle{amsplain}
\bibliography{biblio_trianguline}

\end{document}